\documentclass[10pt]{amsart}
\usepackage{lmodern}
\usepackage[T1]{fontenc}
\usepackage[utf8]{inputenc}
\usepackage{amsmath}
\usepackage{amssymb}
\usepackage{mathrsfs}
\usepackage{comment}
\usepackage{color}
\usepackage{cite}
\usepackage[colorlinks=true,citecolor=red,linkcolor=blue]{hyperref}
\usepackage{enumitem}
\usepackage[mathscr]{eucal}
\usepackage{tikz}
\usepackage{tikzscale}
\usepackage{bbm}

\newtheorem{theorem}{Theorem}[section]

\theoremstyle{definition}
\newtheorem{definition}[theorem]{Definition}
\newtheorem{corollary}[theorem]{Corollary}
\newtheorem{lemma}[theorem]{Lemma}
\newtheorem{proposition}[theorem]{Proposition}

\newtheorem{question}{Question}
\theoremstyle{remark}
\newtheorem{remark}[theorem]{Remark}
\newtheorem{claim}[theorem]{Claim}

\newcommand{\es}{\varnothing}
\newcommand{\seb}{\subseteq}
\newcommand{\setm}{\setminus}
\newcommand{\tup}[1]{\langle#1\rangle}
\newcommand{\ra}{\rightarrow}
\newcommand{\lra}{\leftrightarrow}

\newcommand{\axiom}[1]{\mathsf{#1}}
\newcommand{\ZFC}{\axiom{ZFC}}

\newcommand{\CH}{\axiom{CH}}
\newcommand{\MA}{\axiom{MA}}
\newcommand{\PFA}{\axiom{PFA}}

\newcommand{\cal}[1]{\mathcal{#1}}

\newcommand{\ele}{\prec}

\newcommand{\eleq}{\preceq}
\newcommand{\neleq}{\npreceq}

\newcommand{\sle}{\vartriangleleft}
\newcommand{\sleq}{\trianglelefteq}

\newcommand{\sge}{\vartriangleright}
\newcommand{\sgeq}{\trianglerighteq}
\newcommand{\nsgeq}{\ntrianglerighteq}

\newcommand{\sur}{\twoheadrightarrow}
\newcommand{\ba}{\bar{a}}
\newcommand{\bb}{\bar{b}}
\newcommand{\bc}{\bar{c}}
\renewcommand{\L}{\mathscr{L}}
\newcommand{\R}{\mathscr{R}}
\newcommand{\La}{\hat{\L}}
\newcommand{\Ra}{\hat{\R}}
\newcommand\lex{\mathrm{lex}}

\renewcommand{\P}{\mathbb{P}}
\newcommand{\1}{\mathbbm{1}}
\newcommand{\ck}[1]{\check{#1}}
\newcommand{\dt}[1]{\mathring{#1}}
\newcommand{\tleq}{\sqsubseteq}
\newcommand{\tle}{\sqsubset}
\newcommand{\cc}{\mathfrak{c}}
\newcommand{\sF}{\mathscr{F}}

\DeclareMathOperator{\dom}{dom}
\DeclareMathOperator{\ran}{ran}
\DeclareMathOperator{\al}{ht}
\DeclareMathOperator{\cof}{cof}
\DeclareMathOperator{\coi}{coi}

\title{The class of Aronszajn lines under epimorphisms}

\author{Lucas Polymeris}
\author{Carlos Martinez-Ranero}

\thanks{The authors would like to thank Justin Moore for reading an earlier version of this
  material and pointing out Remark~\ref{rmk:coherent-irrev}}
\thanks{The first named author was supported by ANID-Subdirección de Capital
  Humano/Magíster Nacional/2024 - 22241722}
\thanks{The second named author was partially supported by  Proyecto VRID-Investigación  No. 220.015.024-INV}


\date{\today}

\begin{document}

\begin{abstract}
  A linear order $A$ is called \emph{strongly surjective} if for every
  nonempty suborder $B \eleq A$, there is an epimorphism from $A$ onto
  $B$ (denoted by $B \sleq A$). We show that under $\MA_{\aleph_{1}}$
  there is a strongly surjective Countryman line,
  answering some questions of  Dániel T. Soukup.
  We also study the general structure of
  the class of Aronszajn lines under $\sleq$, and compare it with the
  well known embeddability relation $\eleq$. Under
  $\PFA$ the class of Aronszajn lines and the class of countable linear
  orders enjoy similar nice properties when viewed under the
  embeddability relation; both are well-quasi-ordered and have a finite
  basis.  We show that this analogy does not extend perfectly to the
  $\sleq$ relation; while it is known that the countable linear orders
  are still well-quasi-ordered under $\sleq$, we
  show that already in $\ZFC$ the class of Aronszajn lines has an
  infinite antichain, and under $\MA_{\aleph_{1}}$ an infinite
  decreasing chain as well. We show that some of the analogy
  survives by proving that under $\PFA$, for some carefully
  constructed Countryman line $C$, $C$ and $C^{\star}$ form a $\sleq$-basis
  for the class of Aronszajn lines. Finally we show that this does
  not extend to all uncountable linear orders by proving that
  there is never a finite $\sleq$-basis for the uncountable real orders.
\end{abstract}

\maketitle

\section{Introduction}

Given linear orders $A$ and $B$, an \emph{epimorphism} from $B$ onto
$A$ is any monotone surjective function $f : B \sur A$. Let $A \sleq B$
denote that $A = \es$ or that there is an epimorphism from $B$ onto
$A$. Epimorphisms between linear orders where studied
in~\cite{Landraitis} and later revisited in~\cite{CamerloEtAl2015}
and~\cite{CamerloEtAl2019}. This relation is somewhat a dual of the
\emph{embeddability} relation: say that $A$ \emph{embeds}
into $B$ (denoted by $A \eleq B$) if there is a monotone injective
mapping $f : A \to B$. In this case we also say that $A$ is a
\emph{suborder of} $B$ or that $B$ contains a copy of $A$. Note that
$A \sleq B$ implies $A \eleq B$, while the converse is not
true in general.

In this work we give an analysis of the structure of the class of
Aronszajn lines under $\sleq$, and answer several questions that in
our opinion present themselves naturally once one understands the context and analogies at play. For this it will be particularly important to
have an idea of the development of the $\eleq$ relation on
linear orders, which has been widely studied, and has motivated many deep
and important subjects in Set Theory.  Examples of this are the
formulation of Martin's Axiom ($\MA$), which was conceived by looking
at the solution to Suslin's Problem, and the Proper Forcing Axiom
($\PFA$), whose formulation was influenced by Baumgartner's theorem on
$\aleph_{1}$-dense real orders, and later motivated modern methods
such as the method of Minimal Walks and new set-theoretic hypotheses
of broad interest such as the Open Coloring Axiom and the Mapping Reflection
Principle. We proceed to give
some historical and mathematical context on this development, and in
so doing we shall see that an analogy between the class
of Aronszajn lines and the countable linear orders is suggested under
$\PFA$.

\subsection*{Historical and mathematical context}

Let $\mathcal{R}$ be either $\sleq$ or $\eleq$. Clearly $\mathcal{R}$
is a preorder in the class of linear orders. We say that $A$ and $B$
are \emph{$\mathcal{R}$-equivalent} if $A \mathrel{\mathcal{R}} B$ and
$B \mathrel{\mathcal{R}} A$.  Let $\mathfrak{C}$ be a subclass of the
class of linear orders. We say that $A$ is \emph{$\mathcal{R}$-minimal} in
$\mathfrak{C}$ if $B \mathrel{\mathcal{R}} A$ implies $A
\mathrel{\mathcal{R}} B$ for every $B \in \mathfrak{C}$.  An
\emph{$\mathcal{R}$-basis} for $\mathfrak{C}$ is a subset $F$ of
$\mathfrak{C}$ such that for every $A \in \mathfrak{C}$, there is $X \in
F$ such that $X \mathrel{\mathcal{R}} A$.

We now review the properties of $\eleq$ in some classes of linear orders.
Consider first the class of countable (infinite) linear orders. One
easily sees that $\omega$ and $\omega^{\star}$\footnote{For a linear order $A$,
  $A^{\star}$ denotes its reverse, i.e., the linear order
  defined by $a <_{A^{\star}} b$ iff $b <_{A} a$.}
are the unique (modulo order isomorphism)
$\eleq$-minimal
elements in this class. Recall that $\omega$ is the first infinite ordinal,
isomorphic to the
natural numbers.
It is also easily checked that every infinite linear order contains a
copy of $\omega$ or $\omega^{\star}$, thus $\omega$ and $\omega^{\star}$ form a
$\eleq$-basis for the class of countable linear orders. What more can
be said about the structure of the countable linear orders under $\eleq$? For
starters the class has a top element $\mathbb{Q}$: every countable
linear order embeds into $\mathbb{Q}$. This follows from a theorem of
Cantor, and its proof is maybe the first example of what is now called the
back-and-forth method in logic. Even more can be said; a celebrated
theorem of Laver~\cite{Laver1970}, that answers a conjecture by Fra\"{\i}ssé,
says that the class of countable linear order is well-quasi-ordered by
$\eleq$.

\begin{definition}\label{df:wqo}
  A class $\mathfrak{C}$ is \emph{well-quasi-ordered} by a preorder relation $\mathcal{R}$,
  if it contains no uncountable antichain (an uncountable subset of $\mathfrak{C}$
  of pairwise $\mathcal{R}$-incomparable elements), and no infinite decreasing
  sequence $\cdots A_{2} \mathrel{\mathcal{R}} A_{1} \mathrel{\mathcal{R}} A_{0}$ of not
  $\mathcal{R}$-equivalent members of $\mathfrak{C}$.
\end{definition}

Being well-quasi-ordered can be interpreted as having a structure that is as similar
as possible to the structure of the class of ordinals under $\eleq$.

What about the class of uncountable linear orders? Since being well-quasi-ordered
implies the existence of a finite $\eleq$-basis, let us start from there. The first
difference is that not every uncountable linear order contains a copy
of $\omega_{1}$ or $\omega_{1}^{\star}$, for example $\mathbb{R}$ does not. Is
there a finite $\eleq$-basis for the uncountable real orders? Consistently no:
if $\CH$ holds then every $\eleq$-basis for the uncountable real orders is infinite (see \cite{Baumgartner1982}).

\begin{definition}
  For an infinite cardinal $\kappa$, a nonempty linear order $A$ is called
  \emph{$\kappa$-dense} if it has no left or right endpoint,
  and for each $a <_{A} b$, ${[a,b]}_{A}$ has size $\kappa$.
\end{definition}

An important theorem of Baumgartner~\cite{Baumgartner1973} says that
under $\PFA$ all $\aleph_{1}$-dense real orders are isomorphic.
Noting that each uncountable real order contains an $\aleph_{1}$-dense
suborder, Baumgartner's theorem implies that
under $\PFA$ the uncountable real orders have a single element
basis.
Fix such a set $X$. Does this imply that under $\PFA$, $\omega_{1}$, $\omega_{1}^{\star}$
and $X$ form a $\eleq$-basis for the uncountable linear orders? The answer is no,
because the following object can be proved to exists in $\ZFC$.

\begin{definition}
  An \emph{Aronszajn line} is an uncountable linear order $A$ that does
  not contain copies of $\omega_{1}$, $\omega_{1}^{\star}$, or any uncountable
  set of reals.
\end{definition}

Aronszajn lines can be constructed from Aronszajn trees, objects first
proved to exists by Aronszajn (see~\cite{Kurepa1935}).
Specker~\cite{Specker1949} rediscovered the notion of Aronszajn line,
which is why they are also called Specker orders or Specker types.
Shelah~\cite{Shelah1976} proved the existence of an Aronszajn line $C$,
such that $C$ and $C^{\star}$ do not have uncountable suborders in common,
thus, any $\eleq$-basis for the class of Aronszajn lines contains at least two elements.
He proved this by constructing what is now called a Countryman line, an object
conjectured (to not exist) by Countryman in an unpublished text.

\begin{definition}
  A \emph{Countryman line} is an uncountable linear order such that
  under the product order ($(x,y) \le (x',y')$ iff $x \le_{C} x'$ and $y \le_{C} y'$),
  $C^{2}$ is the union of countably many chains\footnote{A \emph{chain} is
    a set of pairwise comparable elements}.
\end{definition}

It can be proven that any Countryman line is necessarily Aronszajn,
and Shelah conjectured that consistently $C$ and $C^{\star}$ form a
$\eleq$-basis for the Aronszajn lines. Later
this became known as the Five Basis Conjecture: under $\PFA$, If $C$
is any Countryman line, and $R$ any $\aleph_{1}$-dense set of reals,
then $\omega_{1}$, $\omega_{1}^{\star}$, $C$, $C^{\star}$ and $R$ form a $\eleq$-basis for the uncountable linear orders. After about thirty years, this was finally
proved by Moore.

\begin{theorem} (Moore~\cite{Moore2006})
  Assume $\PFA$. If $C$ is any Countryman line, then
  $C$ and $C^{\star}$ form a $\eleq$-basis for the Aronszajn lines.
\end{theorem}

Shortly after, Moore~\cite{Moore2009} also proved that under
$\PFA$ there is also a $\eleq$-top element in the class
of Aronszajn lines.

\begin{theorem}\label{thm:univ-line} (Moore~\cite{Moore2009})
  Assume $\PFA$. There is a universal Aronszajn line $\eta_{C}$.
\end{theorem}

The construction of $\eta_{C}$ is as follows: fix a Countryman line $C$,
and take the lexicographic ordering of
the $\omega$-sequences on $C^{\star} + \{0\} + C$ which are eventually
$0$. Note that by replacing $C$ with $\omega$ in the previous construction
one gets an isomorphic copy of $\mathbb{Q}$. This suggests
an analogy between the countable linear orders and the class of Aronszajn lines
under $\PFA$. Martínez-Ranero extended this
analogy further by proving that the analogous of Laver's theorem holds in
the class of Aronszajn lines.

\begin{theorem}\label{thm:wqo-clines} (Martínez-Ranero~\cite{Martinez})
  Assume $\PFA$. The class of Aronszajn lines is well-quasi-ordered by $\eleq$.
\end{theorem}

Now let us ask about $\sleq$. Again it is not hard to see that
$\omega$, $\omega^{\star}$, $\omega + 1$ and $1 + \omega^{\star}$ form
a $\sleq$-basis for the countable linear orders, and that $\mathbb{Q}$ is
the $\sleq$-top element again. More surprisingly
Landraitis proved in~\cite{Landraitis}
(and later Camerlo, Carroy and Marcone in~\cite{CamerloEtAl2015}) that this class is also well-quasi-ordered by $\sleq$. We believe that
the following questions present themselves naturally.

\begin{question}\label{q:sur-basis}
  Assume $\PFA$. Is there a finite $\sleq$-basis for the class of Aronszajn lines?
\end{question}

\begin{question}\label{q:sur-wqo}
  Assume $\PFA$. Is the class of Aronszajn lines well-quasi-ordered by $\sleq$?
\end{question}

Camerlo, Carroy and Marcone continued\cite{CamerloEtAl2019} their
study of the $\sleq$ relation introducing the following notion.

\begin{definition}
  A linear order $A$ is called \emph{strongly surjective} if
  $B \sleq A$ whenever $B$ is a suborder of $A$.
\end{definition}

Easy examples of strongly surjective orders are $\omega$, $\omega^{\star}$ and
$\mathbb{Q}$. The authors of~\cite{CamerloEtAl2019} ask if there
are uncountable strongly surjective orders. They prove that consistently
yes: if the conclusion of Baumgartner's theorem holds,
then any $\aleph_{1}$-dense set of reals is strongly surjective. They also
proved that any strongly surjective linear order is short\footnote{
  A linear order $L$ is called \emph{short} if $\omega_{1},\omega_{1}^{\star}\neleq L$.}.
Then it is natural to ask for a strongly surjective Aronszajn line.

Shortly after, Soukup~\cite{Soukup} studied these questions in depth. His
results include that
consistently with $\CH$ there are no uncountable strongly
surjective orders, and that consistently there are Aronszajn strongly surjective
orders. In Particular, $\diamondsuit^{+}$ implies the existence of a
strongly surjective lexicographically ordered Suslin tree. He then
asks the following natural questions (see~\cite[Problems~40, 41 and 46]{Soukup}).

\begin{question}\label{q:ma->srs}
  Does $\MA_{\aleph_{1}}$ implies the existence of uncountable
  strongly surjective linear orders?
\end{question}

\begin{question}\label{q:srs-cline}
  Can a strongly surjective linear order be Countryman?
\end{question}

\begin{question}\label{q:srs-real-aronszajn}
  Can there be real and Aronszajn uncountable strongly surjective orders
  simultaneously?
\end{question}

The analogy of the countable linear orders with the class of Aronszajn lines
under $\PFA$ previously mentioned also suggests the following question.

\begin{question}\label{q:univ-srs}
  Assume $\PFA$. Is $\eta_{C}$ (or any other universal Aronszajn line) strongly surjective?
\end{question}

\subsection*{Main results}
Let $A$ be an Aronszajn line, a \emph{decomposition for $A$} is
a $\seb$-increasing sequence $D := \tup{D_{\xi} : \xi < \omega_{1}}$
of countable subsets of $A$,
that covers $A$ and that is continuous, i.e., that at limits is the
union. A well known fact is that any Aronszajn line has size $\aleph_{1}$,
and thus that these decompositions always exist.
The \emph{complementary intervals} of $A \setm D_{\xi}$ are
the intervals of $A$ that are maximal with respect to not having
points in $D_{\xi}$. Since $D_{\xi}$ is countable, and $A$ is
Aronszajn, there are Countably many complementary intervals of $A
\setm D_{\xi}$. Now define $\L(A,D)$ (resp $\La(A,D)$)
to be the set of
$\xi < \omega_{1}$ such that some (resp. every) complementary interval has
a left endpoint. The definition of $\R$ (resp. $\Ra$) is the same
replacing left with right.

With this notation, we say that $A$ is \emph{non-stationary} if there
is a decomposition $D$ for $A$ such that $\L(A)$ and $\R(A)$ are
non-stationary subsets of $\omega_{1}$. In proving Theorem~\ref{thm:univ-line},
Moore \cite{Moore2009} also proved
that under $\MA_{\aleph_{1}}$ any
two $\aleph_{1}$-dense and non-stationary Countryman lines are
isomorphic or reverse isomorphic (under $\PFA$ this was proved before in \cite{ShelahAbraham}). As our first result, we apply this
theorem to prove that under $\MA_{\aleph_{1}}$, any $\aleph_{1}$-dense
and non-stationary Countryman line is strongly surjective.
This answers Questions
\ref{q:ma->srs}, \ref{q:srs-cline} and \ref{q:srs-real-aronszajn}.
For the last one note that one can force $\MA_{\aleph_{1}}$ and the conclusions
of Baumgartner's theorem on $\aleph_{1}$-dense set of reals simultaneously.
Also both follow from $\PFA$.

Similar to Baumgartner's argument \cite{Baumgartner1982} to show that there
are many non-isomorphic Aronszajn line, we show that $\L$ and $\R$
can be used to construct $\sleq$-incomparable Aronszajn lines.

\begin{theorem}
  Suppose that $A$ and $B$ are Aronszajn lines with respective
  decompositions $D$ and $E$. If $A \sgeq B$ then
  $\La(A,D) \setm \La(B,E)$ and $\Ra(A,D) \setm \Ra(B,E)$
  are non-stationary subsets of $\omega_{1}$.
\end{theorem}

% The idea is that if $f : A \sur B$ is an epimorphism, then
% modulo a club subset of levels, $f$ must map left (resp. right)
% endpoints of the complementary intervals of $A \setm D_{\xi}$ to
% left endpoints of the complementary intervals of $B \setm E_{\xi}$.
We then apply this theorem to construct an infinite
$\sleq$-antichain of Countryman lines, giving a negative answer to Question~\ref{q:sur-wqo}.

Fix $C$ and $C'$ be $\aleph_{1}$-dense Countryman lines. For $C \sgeq C'$ to
hold, an obvious necessary condition is that $C' \eleq C$. And we
know that under $\MA_{\aleph_{1}}$, either $C' \eleq C$, or $C' \eleq C^{\star}$.
The next theorem shows that under $\MA_{\aleph_{1}}$, this together with
the previous theorem, are roughly the unique limitations if
restricted to $\aleph_{1}$-dense lines.

\begin{theorem}\label{thm:main2}
  Assume $\MA_{\aleph_{1}}$. Let $C' \eleq C$ be $\aleph_{1}$-dense
  Countryman lines with respective decompositions $D$ and $D'$.
  If for some club $E \seb \omega_{1}$,
  $\L(C,D) \cap E \seb \La(C',D')$ and $\R(C,D) \cap E \seb \Ra(C',D')$,
  then $C \sgeq C'$.
\end{theorem}

The proof consists of constructing a ccc forcing notion, strongly
based in Moore's forcing
in \cite{Moore2009}, which under the correct hypotheses
introduces an epimorphism from $C$ onto $C'$.

We immediately proceed to give three applications of this theorem.
First, an alternative
proof that any $\aleph_{1}$-dense and non-stationary Countryman line is
strongly surjective under $\MA_{\aleph_{1}}$. Secondly the construction of
an infinite $\sle$-decreasing chain of Countryman lines
under $\MA_{\aleph_{1}}$, thus giving
an alternative answer to Question~\ref{q:sur-wqo}. And finally, in conjunction
with a theorem of Moore, we show that it implies that the class
of Aronszajn lines has a two element $\sleq$-basis under $\PFA$.
Thus answering affirmatively Question~\ref{q:sur-basis}.
We also show that this cannot be extended to all
uncountable linear orders, namely we prove that there is never
a finite $\sleq$-basis for the uncountable linear orders.

The only question that remains open is Question~\ref{q:univ-srs}.

\subsection*{Organization}
In Section~\ref{sec:prelim} we develop the mathematical background that is
needed later.
In Section~\ref{sec:srs} we show that under
$\MA_{\aleph_{1}}$ there are many strongly surjective Aronszajn lines.
In particular we prove that $\MA_{\aleph_{1}}$ implies the existence
of a strongly surjective Countryman line.
In Section~\ref{sec:decompositions}
we develop the concept of decompositions for Aronszajn lines, and
construct Countryman lines with particular configurations of $\L$ and
$\R$.
In Section~\ref{sec:infinite-antichain} we construct an infinite
$\sleq$-antichain of Countryman lines under $\ZFC$.
In Section~\ref{sec:infinite-chain} we construct an infinite
$\sle$-decreasing chain of Countryman lines under $\MA_{\aleph_{1}}$. It is
in this section that we present our main Forcing.
In Section~\ref{sec:basis} we apply the forcing to show that
under $\PFA$, the class of Aronszajn lines has a two element $\sleq$-basis.
We also show that there is never a $\sleq$-finite basis for the uncountable real
orders.
In Section~\ref{sec:irreversible} we take a detour from the $\sleq$ relation
to answer a question on Countryman lines that seemed natural to us,
but does not seem to have been considered in the literature.

\subsection*{Notation and prerequisites}

We assume familiarity with standard terminology and notations of Set
Theory. We refer to~\cite{Jech2003} and~\cite{Kunen1980} for undefined
notions. Particularly relevant is the notion of trees and their interplay with
linear orders; the material in~\cite{Todorcevic1984} should prove
helpful for a deeper understating.
Regarding set-theoretic hypotheses, we will deal primarily
with $\MA_{\aleph_{1}}$ and $\PFA$, but with the exception
of Section~\ref{sec:infinite-chain} in which knowledge of the
statement of $\MA_{\aleph_{1}}$ is necessary, these can be used
as black boxes in the rest of the text. Knowledge of the
specifics of the method of Forcing will be needed
in Sections~\ref{sec:infinite-chain} and \ref{sec:irreversible}.

A \emph{sequence} is any function with domain an ordinal. We let
$\tleq$ denote sequence extension and $\tle$ proper sequence extension.
If $S$ is
a set of sequences with ranges in a linearly ordered set $A$, then
$S$ is naturally linearly ordered by the lexicographic ordering, given by
$t <_{\lex} s$ if either $t \tle s$, or $t(\xi) <_{A} s(\xi)$
where $\xi$ is the first ordinal at which they differ.

When talking about orders we will usually not specify the order relation,
and use subscripts to do it as it was done above in the definition
of $\le_{\lex}$. A plain $<$ usually refers to the standard order
on ordinals.
When thinking of $A \times B$ as a linear order, we give it the
lexicographically order, that is $A$ copies of $B$.
This is somewhat nonstandard since then
$\alpha \times \beta \cong \beta \cdot \alpha$, when thinking of
ordinal multiplication. If for all $a \in A$, $L_{a}$ is a
linear order, $\sum_{a \in A}L_{a}$ denotes the order consisting
of lifting each $a \in A$ to a copy of $L_{a}$. Formally
this is the lexicographic ordering of
$\bigcup_{a \in A}\{a\} \times L_{a}$. Note that $A \times B = \sum_{a \in A}B$.

An \emph{interval}
is any nonempty set $I \seb A$ that is \emph{convex} in the sense that
if $a <_{A} b$ are in $I$, then ${[a,b]}_{A} \seb I$, $I$ is called
\emph{trivial} if $|I| = 1$. Examples of intervals
are intervals of the form ${[a,b]}_{A}$, ${\left]a,b\right[}_{A}$,
${\left[a,b\right[}_{A}$ and ${\left]a,b\right]}_{A}$ which are
defined as usual. We usually drop the subscripts if this causes no confusion.
At some points we will allow $a$ and $b$ to be $+\infty$/$-\infty$,
giving the obvious interpretation. For example
$\left[a,+\infty \right[$ stands for $\{b \in A : a <_{A}b\}$.
Note that these are not necessarily all
the intervals of $A$. For example $\mathbb{Q} \cap [0,\sqrt{2}]$ is an interval
of $\mathbb{Q}$ that cannot be written this way.

If $A$ is linearly ordered we say that $A$ has a left
(resp. right) \emph{endpoint} if it has a minimum (resp. maximum).

\section{Aronszajn and Countryman lines}\label{sec:prelim}

In this section we review some known facts about Aronszajn lines, and
some of their known properties under axioms such as $\MA_{\aleph_{1}}$
and $\PFA$.

\begin{definition}
  For a linear order $L$ and $X \seb L$, the \emph{complementary intervals}
  of $L \setm X$ are the intervals of $L$
  which are maximal with respect to not
  having points in $X$.
\end{definition}

Note that the complementary intervals of $L \setm X$ always form a
partition of $L \setm X$ into intervals, and that if $L$ is Aronszajn
and $X$ is countable, then this partition is countable.
Following~\cite{TodorcevicWalks} we make the
following definition.

\begin{definition}\label{df:decomposition}
  Let $A$ be an Aronszajn line. A \emph{decomposition for} $A$ is
  a $\seb$-increasing and continuous (at limits is the union) sequence
  $\tup{D_{\xi} : \xi < \omega_{1}}$ of countable subsets of $A$
  that covers $A$.
\end{definition}

An important fact is that any Aronszajn line has size $\aleph_{1}$
(see for example~\cite[Corollary~4.3]{Baumgartner1982}),
therefore decompositions always exist. The following is one of the
standard ways to go from an Aronszajn line to an Aronszajn tree.
It is attributed to Kurepa in \cite{Todorcevic1984}.

\begin{lemma}
  If $A$ is Aronszajn and $\tup{D_{\xi} : \xi < \omega_{1} }$ is a decomposition
  for $A$, then $T := \bigcup_{\xi < \omega_{1}}
  \{I : I \text{ is a complementary interval of $A \setm D_{\xi}$}\}$
  ordered by reverse inclusion is an Aronszajn tree.
\end{lemma}

Following~\cite{Moore2009} and~\cite{ShelahAbraham} we make the
following definition.

\begin{definition}\label{df:normal}
  An Aronszajn line $A$ is called
  \emph{non-stationary} if it has a decomposition $\tup{D_{\xi} : \xi <
    \omega_{1}}$ such that for every $\xi < \omega_{1}$, every
  complementary interval of $A \setm D_{\xi}$ has no endpoints.  $A$
  is called \emph{normal} if it is non-stationary and
  $\aleph_{1}$-dense.
\end{definition}

The following, which also seems folklore, will be particularly
relevant.  We do not know of a proof in the literature, but such a
proof can easily be extracted from the material in Section~\ref{sec:decompositions}.

\begin{lemma}\label{lem:normal-suborder}
  Every Aronszajn line contains a normal suborder.
\end{lemma}

The following summarizes the most important properties of Countryman lines,
which where defined in the introduction.
For a proof see~\cite[Theorem~5.4]{Todorcevic1984}.

\begin{lemma}\label{lem:cline-fund}
  If $C$ is a Countryman line then,
  \begin{itemize}\itemsep0em
    \item $C$ is Aronszajn.
    \item $C$ contains no two uncountable reverse isomorphic suborders, that
    is, for no uncountable $X$, both $X \eleq C$ and $X \eleq C^{\star}$ hold.
    \item For every $n \ge 1$, $C^{n}$ (in the product order)
    is the union of countably many chains.
  \end{itemize}
\end{lemma}

It is unclear to the author who first proved the following.
For a proof see \cite[2.1.12 and 2.1.13]{TodorcevicWalks}.

\begin{theorem}\label{thm:clines-equiv}
  Assume $\MA_{\aleph_{1}}$. The class of Countryman
  lines has exactly two $\eleq$-equivalence classes. In particular
  if $C$ is any Countryman line, then $C$ is a $\eleq$-minimal uncountable
  order and $\{C,C^{\star}\}$ is a $\eleq$-basis for the Countryman lines.
\end{theorem}

The following is a strong form of the previous theorem (recall
Lemma~\ref{lem:normal-suborder}). It appears in \cite{ShelahAbraham}
under the stronger assumption of $\PFA$. Complete details can be found
in \cite{Moore2009}, where only $\MA_{\aleph_{1}}$ is used.

\begin{theorem}\label{thm:moore-t1.1} (Abraham-Shelah, Moore)
  Assume $\MA_{\aleph_{1}}$.
  Any two normal Countryman lines are either isomorphic or reverse isomorphic.
\end{theorem}

For the rest of this section we fix a Countryman line $C$. Recall
the definition of $\eta_{C}$ from the introduction.

\begin{theorem}\label{thm:univ-aline-strong} (Moore~\cite{Moore2009})
  Assume $\PFA$. For every Aronszajn line $A$,
  $A \eleq \eta_{C}$. Moreover, either
  $A$ is $\eleq$-equivalent to $\eta_{C}$, or $A$ contains an interval
  $\eleq$-equivalent to $C$ or to $C^{\star}$.
\end{theorem}

Making the analogy with the scattered linear orders (linear orders without copies
of $\mathbb{Q}$), Moore calls an Aronszajn line
\emph{fragmented} if it does not contain a copy of $\eta_{C}$.
The following definitions
are due to Martínez-Ranero~\cite{Martinez}. Here we assume that $C$ is $\aleph_{1}$-dense.

For $\alpha < \omega_{2}$ recursively define Aronszajn lines
$D_{\alpha}^{+}$ and $D_{\alpha}^{-}$ as follows.
\begin{itemize}\itemsep0em
  \item $D_{0}^{+} := C$ and $D_{0}^{-} := C^{\star}$.
  \item $D_{\alpha+1}^{+} = C \times D_{\alpha}^{-}$ and
  $D_{\alpha+1}^{-} := C^{\star} \times D_{\alpha}^{+}$.
  \item If $\alpha$ is limit let
  $D_{\alpha}^{+} := \sum_{x \in C}A_{x}$ such that each $A_{x}$ is
  $D_{\xi}^{-}$ for some $\xi < \alpha$, and such that
  for each $\xi < \alpha$, $\{x \in C : A_{x}= D_{\xi}^{-}\}$ is dense
  in $C$. $D_{\alpha}^{-}$ is defined similarly.
\end{itemize}

\begin{lemma}\label{lem:D_a} (Martínez-Ranero~\cite{Martinez})
  Assume $\MA_{\aleph_{1}}$. For any fragmented Aronszajn line $A$,
  there is $\alpha < \omega_{2}$ such that
  either $A$ is $\eleq$-equivalent to $D_{\alpha}^{-}$ or
  $D_{\alpha}^{+}$, or $A \ele D_{\alpha}^{-}$ and $A \ele D_{\alpha}^{+}$.
\end{lemma}

\section{Strongly surjective Aronszajn lines}\label{sec:srs}

In this chapter we show that under $\MA_{\aleph_{1}}$ there are many
strongly surjective Aronszajn lines; in particular we deduce
from Theorem~\ref{thm:moore-t1.1} that any normal
Countryman line is strongly surjective under $\MA_{\aleph_{1}}$.
In Section~\ref{sec:infinite-chain} we
show that this can be achieved directly by forcing also.

\begin{lemma}\label{lem:c-prod}
  If $A$ is either countable (but nonempty) or Aronszajn, and $B$ is
  any normal Aronszajn line, then $A \times B$ is also a normal
  Aronszajn line.
\end{lemma}
\begin{proof}
  Clearly $A \times B$ is an $\aleph_{1}$-dense Aronszajn line.
  It remains to prove that it is also non-stationary.
  Let $\tup{D_{\xi} : \xi < \omega_{1}}$ witness the non-stationarity
  of $B$, and let $\{a_{\xi} : \xi < \omega_{1}\}$ enumerate $A$ (with
  repetitions allowed). We claim that $\tup{U_{\xi} : \xi < \omega_{1}}$
  where $U_{\xi} := \{a_{\eta} : \eta <
  \xi\} \times D_{\xi}$,  witnesses the non-stationarity of $A \times B$.

  It should be clear that $U$ is increasing, continuous, covers $A \times B$ and
  consists of countable subsets of $A \times B$. Fix $\xi < \omega_{1}$,
  $I$ a complementary interval of $(A\times B) \setm U_{\xi}$,
  and $(a,b)$ an element of $I$. Also let $J$ be the complementary interval
  of $B \setm D_{\xi}$ in which $b$ is. Since $J$ has no endpoints
  there are $b',b'' \in J$ such that $b' <_{B} b <_{B} b''$. But then
  $(a,b') <_{A \times B} (a,b) <_{A\times B} (a,b'')$ and $(a,b'),(a,b'') \in I$.
  We conclude that $I$ has no endpoints either.
\end{proof}

\begin{corollary}\label{cor:cline-srs}
  Assume $\MA_{\aleph_{1}}$. If $C$ is a normal Countryman line, then $A
  \times C \cong C$ for every nonempty $A \eleq C$. In particular $C$ is
  strongly surjective.
\end{corollary}
\begin{proof}
  If $A \eleq C$ is nonempty, then $A \times C$ is Countryman and it is
  normal by Lemma~\ref{lem:c-prod}. Then by Theorem~\ref{thm:moore-t1.1}, $A \times C$
  is isomorphic to $C$ or to $C^{\star}$, but the later is impossible
  by Lemma~\ref{lem:cline-fund}.
\end{proof}

As mentioned in the introduction, this answers positively
Soukup questions \ref{q:ma->srs}, \ref{q:srs-cline} and
\ref{q:srs-real-aronszajn}.

In view of Corollary~\ref{cor:cline-srs}, one could ask if in fact every normal
Aronszajn line is strongly surjective under $\MA_{\aleph_{1}}$, or
under stronger assumptions, however we can easily see that this is not
the case: if $C$ is a normal Countryman line, then $C + C^{\star}$ is
a normal Aronszajn line that is not strongly surjective: one
easily sees that $C + C^{\star} \nsgeq C$. Although
normality is not enough, we do have many examples of non-Countryman
Aronszajn lines that are strongly surjective under $\MA_{\aleph_{1}}$,
as we show in what follows. Recall the following
fact from~\cite{CamerloEtAl2019}.

\begin{lemma}
  Every strongly surjective linear order is short\footnote{Recall that
    a linear order is called \emph{short} if it does not contain a copies
    of $\omega_{1}$ or $\omega_{1}^{\star}$.}.
\end{lemma}

In~\cite[Corollary~2.14]{CamerloEtAl2019}, they prove that the
product of strongly surjective linear orders is again strongly surjective.
It is not hard to see that they actually prove the following stronger form.

\begin{lemma}\label{lem:sur-both-sides} (Camerlo et al.~\cite{CamerloEtAl2019})
  If $A$ and $B$ are short,
  $A' \sleq A$ and $B' \sleq B$,
  then $A' \times B' \sleq A \times B$. In particular
  $A \times B$ is strongly surjective whenever $A$ and $B$ are.
\end{lemma}

\begin{definition}
  $L$ is said to be \emph{open strongly surjective} if for all
  nonempty $A \seb L$ there is an epimorphism $f : L \sur A$ such that
  for all $a \in A$, $f^{-1}(a)$ has no endpoints.
\end{definition}

\begin{definition}
  Let $I$ be any dense linear order, an \emph{$I$-mixed sum} is any
  order of the form $\sum_{i \in I}A_{i}$ where each $A_{i}$ is nonempty
  and such that for all $i \in I$, $\{j \in I : A_{i} = A_{j}\}$ is
  dense in $I$.
\end{definition}

The following lemma is already present in~\cite{Soukup} for the
case $I = \mathbb{Q}$.

\begin{lemma}\label{lem:srs-mix}
  If $I$ is an open strongly surjective dense order without endpoints,
  then any $I$-mixed sum of nonempty strongly surjective orders
  is again strongly surjective.
\end{lemma}
\begin{proof}
  Let $A = \sum_{i \in I}A_{i}$ be the mixed sum and let $B \seb A$ be
  nonempty, we show that $A \sgeq B$. First note that
  $B$ can be written as $\sum_{i \in I'}A_{i}'$ where $I' \seb I$ and
  $A_{i}' \seb A_{i}$ are all nonempty. Since each $A_{i}$ is strongly
  surjective, $A \sgeq B$ follows from
  $A \sgeq \sum_{i \in I'}A_{i}$.
  Let $f : I \sur I'$ be an epimorphism with open preimages and for
  each $i \in I'$, let $I_{i} := f^{-1}(i)$. It is enough to show that
  for each $i \in I'$, $\sum_{j \in I_{i}}A_{j} \sgeq A_{i}$.

    Fix $i \in I'$. Using the fact that $I_{i}$ has countable
  cofinality and coinitiality, plus the density of $\{j : A_{j} =
  A_{i}\}$, pick $\{j_{z} : z \in \mathbb{Z}\}$ a copy of $\mathbb{Z}$
  coinitial and cofinal in $I_{i}$ such that for all $z$, $A_{j_{z}} =
  A_{i}$. From this one easily sees that $\sum_{j \in I_{i}} A_{j} \sgeq
  \cdots + 1 + A_{i} + 1 + A_{i} + 1 \cdots$, therefore it is enough to prove
  that $\cdots + 1 + A_{i} + 1 + A_{i} + 1 \cdots \sgeq A_{i}$.

  This is done by cases depending on the
  cofinality and coinitiality of $A_{i}$. If both the
  cofinality and the coinitiality of $A_{i}$ are $1$ this is trivial. We
  do the case when the coinitiality is $1$ and the cofinality is $\omega$,
  the case when $(\coi(A_{i}),\cof(A_{i})) = (\omega^{\star},1)$ is symmetric,
  and when it is $(\omega^{\star},\omega)$ is very similar.
  Recall that $A_{i}$, being strongly surjective, is short, and thus these
  are all the possible cases.

  Let $\tup{a_{n} : n < \omega}$ be increasing and cofinal in $A$,
  and such that $a_{0}$ is the left endpoint of $A$. Note that formally
  we think of $\cdots + 1 + A_{i} + 1 + A_{i} + 1 \cdots$ as
  \[\bigcup_{n \in \mathbb{Z}}\{2n\}\times 1 \cup \bigcup_{n \in \mathbb{Z}}\{2n+1\}\times A_{i}\]
  with the lexicographical ordering.
  Then
  \[f(n,x) :=
    \begin{cases}
      a_{0}   &\text{if $n < 0$,}\\
      a_{k}   &\text{if $0 \le n = 2k$,}\\
      \max_{A_{i}}(a_{k},\min_{A_{i}}(x,a_{k+1})) &\text{if $0 \le n = 2k + 1$.}
    \end{cases}
  \]
  witnesses the fact that
  $\cdots + 1 + A_{i} + 1 + A_{i} + 1 \cdots \sgeq A_{i}$
\end{proof}

\begin{corollary}
  Assume $\MA_{\aleph_{1}}$. For every family $\cal{F}$ of
  at most $\omega_{1}$ strongly surjective orders, there is a strongly
  surjective order of size $\sup\{\omega_{1},\sup\{|A| : A \in \cal{F}\}\}$
  that contains a copy of every order in $\cal{F}$.
\end{corollary}
\begin{proof}
  Simply note that Corollary~\ref{cor:cline-srs} implies that every normal
  Countryman line is in fact open strongly surjective under
  $\MA_{\aleph_{1}}$. Then if $C$ is a normal Countryman line,
  any $C$-mixed sum of $\cal{F}$ works.
\end{proof}

Note that as in~\cite{Soukup}, using $\mathbb{Q}$ instead of $C$
gives a $\ZFC$ proof of the previous corollary when $|F| \le \aleph_{0}$.

Recall the definition of $D_{\alpha}^{+}$ and $D_{\alpha}^{-}$ from
 Lemma~\ref{lem:D_a}.  One easily sees that Lemma~\ref{lem:srs-mix} and
 Lemma~\ref{lem:sur-both-sides} imply the following, which shows that under
$\MA_{\aleph_{1}}$ there are many really different (not $\eleq$-equivalent)
Aronszajn lines.

\begin{corollary}\label{lem:D_a_srs}
  Assume $\MA_{\aleph_{1}}$. If the Countryman line used to build them is
  normal, then for every $\alpha < \omega_{2}$, $D_{\alpha}^{+}$ and
  $D_{\alpha}^{-}$ are strongly surjective.
\end{corollary}

For the rest of this section fix $C$ a normal Countryman line.
While we do not have an answer to Question~\ref{q:univ-srs}, we
do have the following.

\begin{lemma}\label{lem:sur_frag}
  Assume $\PFA$. For every fragmented Aronszajn line $A$, $\eta_{C} \sgeq A$.
\end{lemma}
\begin{proof}
  By Lemma~\ref{lem:D_a} we know that for any fragmented Aronszajn line
  there is a least $\alpha$ such that $A \eleq D_{\alpha}^{-}$
  or $A \eleq D_{\alpha}^{+}$. Thus by Lemma~\ref{lem:D_a_srs}, it is enough
  to prove that $\eta_{C} \sgeq D_{\alpha}^{+},D_{\alpha}^{-}$ for every
  $\alpha < \omega_{1}$.

  We do it by induction on $\alpha$.
  For $\alpha = 0$ we need to show that $C,C^{\star} \sleq
  \eta_{C}$. We do it for $C$, the other case is identical.
  Note that $\eta_{C} \sgeq {(C^{\star} + 1 + C)}^{2}
  \sgeq \mathbb{Z} \times (1+C) \cong C$, where the first inequality is
  witnessed by $\tup{x_{n} : n < \omega} \mapsto (x_{0},x_{1})$, the
  second follows from the strong surjectivity of $C$ and Lemma~\ref{lem:sur-both-sides},
  and the isomorphism follows from Theorem~\ref{thm:moore-t1.1}. Also a direct epimorphism
  from $\mathbb{Z} \times (1+C)$ onto $C$ is not hard to construct.

  If $\alpha > 0$, then
  regardless of whether $\alpha$ is limit or successor and
  whether $i = +$ or $i = -$,
  $D_{\alpha}^{i}$ can be written as $\sum_{x \in C}A_{x}$ where each
  $A_{x}$ is in $\{D_{\xi}^{x} : \xi < \alpha, x \in \{-,+\}\}$.
  Also note that
  ${(C^{\star} + 1 + C)}^{2} \times \eta_{C} \cong \eta_{C}$. Thus
  it is enough to prove that
  \[{(C^{\star} + 1 + C)}^{2} \times \eta_{C} \sgeq \sum_{x \in C}A_{x}.\]

  Let $f : {(C^{\star} + 1 + C)}^{2} \sur C$ be an epimorphism,
  and for each $x \in C$, let $I_{x} = f^{-1}(x)$. Now from
 Lemma~\ref{lem:sur-both-sides}, we know that
  $I_{x} \times \eta_{C} \sgeq \eta_{C}$, and by induction we obtain
  $\eta_{C} \sgeq A_{x}$. Finally,
  \[{(C^{\star} + 1 + C)}^{2} \times \eta_{C}
    = \sum_{x \in C}f^{-1}(x)\times\eta_{C} \sgeq \sum_{x \in C}A_{x} = D_{\alpha}^{i}.\]
\end{proof}

\section{Aronszajn line decompositions}\label{sec:decompositions}

In this section we develop the notion of decompositions of Aronszajn lines (recall Definition~\ref{df:decomposition}). In particular
we construct several Aronszajn lines with decompositions that achieve
specific properties that will be used to study the $\sleq$ relation
in Sections \ref{sec:infinite-antichain}, \ref{sec:infinite-chain} and \ref{sec:basis}.

\begin{lemma}\label{lem:dec-at-clubs}
  Let $A$ be an Aronszajn line, and $D$, $D'$ be two decompositions for $A$.
  Then there is a club $E \seb \omega_{1}$
  such that for all $\nu \in E$, $D_{\nu} = D'_{\nu}$.
\end{lemma}
\begin{proof}
  Assume that $A$ is $\omega_{1}$ as a set and let $E \seb \omega_{1}$
  be a club of limit ordinals closed under $f$ and $g$ where
  \[ f(\xi) := \min\{\alpha : \xi \in D_{\alpha}\}\text{ and }
    g(\xi) := \sup\{\eta + 1 : \eta \in D_{\xi}\}.\]
  We show that if $\nu \in E$, then $D_{\nu} = \nu$.
  First observe that for all $\xi < \nu$,
  $D_{\xi} \seb g(\xi) < \nu$, and since decompositions
  are the union at limit ordinals, we conclude that $D_{\nu} \seb \nu$.
  Suppose that for some $\xi$, $\xi \in \nu \setm D_{\nu}$.
  Then $f(\xi) < \nu$, which implies
  that for some $\alpha < \nu$, $\xi \in D_{\alpha}$. But
  then since $D$ is increasing, $\xi \in D_{\nu}$, which is
  a contradiction.

  Using an analogous argument for $D'$, and intersecting the resulting clubs,
  we get a club $E$ such that for all $\nu \in E$, $D_{\nu} = \nu = D'_{\nu}$.
\end{proof}

Let us fix an Aronszajn line $A$.
If $D = \tup{D_{\xi} : \xi < \omega_{1}}$ is a decomposition for
$A$, let $\L(A,D)$ denote the set of $\xi < \omega_{1}$ such that some
complementary interval of $A \setm D_{\xi}$ has a left endpoint. And
let $\La(A,D)$ denote the set of $\xi < \omega_{1}$ such that
every complementary interval of $A \setm D_{\xi}$ has a left endpoint.
Let $\R(A,D)$ and $\Ra(A,D)$ be analogous replacing left by right.
Recall Definition~\ref{df:normal}. With this notation one easily sees that the definition
of non-stationary in Aronszajn lines, says exactly that
for some decomposition $D$, $\L(A,D) = \es = \R(A,D)$. This, together
with Lemma~\ref{lem:dec-at-clubs}, easily implies the following, which also
explains the choice of naming in that definition.

\begin{proposition}\label{prop:nonstat-eq}
  $A$ is non-stationary iff for every decomposition $D$,
  $\L(A,D)$ and $\R(A,D)$ are non-stationary subsets of $\omega_{1}$.
\end{proposition}

We dedicate the rest of this section to construct Aronszajn lines
with specific configurations of $\L$ and $\R$.

Recall that for functions with the same domain,
$f =^{*} g$ denotes that $\{x \in \dom(f) : f(x) \neq g(x)\}$ is finite.
We also say that a function $f$ is \emph{finite-to-one} if $f^{-1}(y)$
is finite for every $y \in \ran(f)$. If $T$ is a set of
sequences (i.e., functions with domain an ordinal), $T$ is called
\emph{coherent} if for all $t,s \in T$ with $\dom(t) \le \dom(s)$,
$t =^{*} s|_{\dom(t)}$.

Let $\tup{f_{\alpha} : \alpha < \omega_{1}}$ be a coherent
sequence of finite-one-functions $f_{\alpha} : \alpha \to \omega$.
It is well known that such a sequence exists in $\ZFC$, and that
$T := \{t \in \alpha^{\omega} : \alpha < \omega_{1}, t =^{*} f_{\alpha}\}$
is an Aronszajn tree. The following is a theorem
of Todorcevic (see~\cite{Todorcevic1987}). For a somewhat simpler
proof see \cite{KomjathTotik}.

\begin{lemma}\label{lem:coherent-fin-countryman}
  $T$ is Countryman with the natural lexicographic ordering.
\end{lemma}

Let $\Lambda$ denote the set of countable limit ordinals,
and $\Lambda'$ those that are limit of elements in $\Lambda$. Both these
sets are club. Define $A^{+} := A_{0} \cup A_{1}$ where
$A_{0} := \{t \in T : \dom(t) \in \Lambda\}$ and
$A_{1} := \{t^{\frown}\tup{\omega} : t \in A_{0}\}$. We understand $A^{+}$
as a linear order with the lexicographic ordering inherited from $T$,
so $A^{+}$ is an Aronszajn line.

\begin{lemma}\label{lem:Aplus-dense}
  $A^+$ is $\aleph_{1}$-dense.
\end{lemma}
\begin{proof}
  We show that if $t <_{\lex} s$ are in $A^{+}$, then
  $[t,s]_{\lex} \cap A^+$ is uncountable. The proof that $A^+$ has
  no endpoints is similar. First observe that it is enough
  to find $u \in T$ such that $t <_{\lex} u <_{\lex} s$ and
  is $T$-incomparable with $s$, because then $\{t \in A^{+} : u \tleq t\}$
  is uncountable and contained in $[t,s]_{\lex}$. This is clear
  if $t \not\tle s$ (simply take $u := t)$, thus assume that $t \tle s$. Since
  the elements of $A_{1}$ are all maximal in $T$, we have that
  $t \in A_{0}$.
  Now observe that
  $\dom(t) + \omega \le \dom(s)$, and since $s$ is finite-to-one, for
  some $n < \omega$, $s(\dom(t) + n + 1) \neq 0$. Then
  $u := (s|_{\dom(t) + n})^{\frown}\tup{0}$ is as wanted.
\end{proof}

\begin{theorem}\label{thm:L_R}
  For every $X,Y \seb \omega_{1}$ there is an $\aleph_{1}$-dense Aronszajn
  line $A$, and decomposition $D$ for $A$, such that,
  \begin{itemize}\itemsep0em
    \item $\La(A,D) = \L(A,D) = X$,
    \item $\Ra(A,D) = \R(A,D) = Y$ and,
    \item For every $\xi < \omega_{1}$, no complementary
    interval of $A \setm D_{\xi}$ is a singleton.
  \end{itemize}
\end{theorem}
\begin{proof}
  Fix $X$ and $Y$ and let $\tup{\lambda_{\xi} : \xi < \omega_{1}}$
  be the increasing enumeration of $\Lambda'$. We
  will construct $A$ as a subtree of $A^{+}$ by removing the unwanted
  endpoints. For this let $A$ be the union of
  $\{t \in A_{0} : \dom(t) \in \Lambda \setm \Lambda ' \}$,
  $\{t \in A_{0} : \dom(t) = \lambda_{\xi}, \xi \in X \}$ and
  $\{t^{\frown}\tup{\omega} \in A_{1} : \dom(t) = \lambda_{\xi}, \xi \in Y\}$.

  The $\aleph_{1}$-density of $A$ follows from the fact that
  in the proof of Lemma~\ref{lem:Aplus-dense}, the $u$ satisfies
  that $\{t \in A : u \tleq t\}$ is also uncountable. To see this note that
  for every $t \in T$, $\{s \in T : t \tle s, \dom(s) \in \Lambda \setm \Lambda'\}$
  is uncountable.

  For $\xi < \omega_{1}$, let $D_{\xi} := \{t \in A : \dom(t) < \lambda_{\xi}\}$.
  Clearly $D := \tup{D_{\xi} : \xi < \omega_{1}}$ is a decomposition for $A$.
  We claim that for every $\xi < \omega_{1}$, the complementary
  intervals of $A \setm D_{\xi}$ are all of the form
  $\{t \in A : t_{0} \tleq t\}$
  for some $t_{0} \in A_{1}$ with $\dom(t_{0}) = \lambda_{\xi}$. Before
  doing this note that this implies the desired result, since
  this interval has a left endpoint iff
  $t_{0} \in A$ iff $\xi \in X$, and a right one iff
  ${t_{0}}^{\frown}\tup{\omega} \in A$ iff $\xi \in Y$.

  Fix $\xi < \omega_{1}$ and $t <_{\lex} s$ in $A \setm D_{\xi}$. Also
  let $t_{0} := t|_{\lambda_{\xi}}$ and $t_{1} := s|_{\lambda_{\xi}}$. It is enough
  to prove that if $t_{0} \neq t_{1}$, then they are not in the same complementary
  interval of $A \setm D_{\xi}$. Assume that $t_{0} \neq t_{1}$, so that in
  fact $t_{0} <_{\lex} t_{1}$. Let $\delta < \lambda_{\xi}$ be the least
  such that $t_{0}(\delta) \neq t_{1}(\delta)$. Now
  $u:= t_{0}|_{\delta + \omega}$ satisfies that $t <_{\lex} u <_{\lex} s$,
  $\dom(u) < \lambda_{\xi}$, $\dom(u) \notin \Lambda'$, so $u \in A$
  witnesses that $t_{0}$ and $t_{1}$ are in distinct complementary intervals
  of $A \setm D_{\xi}$.
\end{proof}

\begin{corollary}\label{cor:countryman_L}
  If $Y = \es$, then we can require $A$ in the conclusion
  of Theorem~\ref{thm:L_R} to be Countryman.
\end{corollary}
\begin{proof}
  Simply note that in the proof of Theorem~\ref{thm:L_R} if $Y = \es$ then $A
  \seb T$. So this follows from Lemma~\ref{lem:coherent-fin-countryman}.
\end{proof}

\begin{lemma}\label{lem:countryman_iff_special}
  $A^{+}$ is Countryman iff $T$ is
  special\footnote{A tree $T$ is called \emph{special} if it is the union
    of countably many antichains. Equivalently, if there
    is $f : T \to \omega$ such that $f(t) \neq f(s)$ whenever $t <_{T} s$.
  }.
\end{lemma}
\begin{proof}
  $(\ra)$. Let $T_{s}$ be the set of $t \in T$ that have successor
  length, and for $t \in T_{s}$ let $l(t) := t(\alpha)$ where
  $\dom(t) = \alpha+1$. Since $T$ consists of finite-to-one functions, one
  sees that $t \mapsto (l(t), |\{\xi < \dom(t) : t(\xi) = l(t)\}|)$
  witnesses that $T_{s}$ is special. Thus it is enough to prove that
  $A_{0}$ is special.

  Let $c : A^{+} \times A^{+} \to \omega$
  witness that $A^{+}$ is Countryman, i.e., for every $n < \omega$,
  $f^{-1}(n)$ is  a chain in the product order of $A^{+} \times A^{+}$.
  We define
  $f : A_{0} \to \omega$ by letting $f(t) := c(t^{\frown}\tup{\omega},t)$,
  and prove that for $t \tle s$, $f(t) \neq f(s)$.

  Assume $t \tle s$ and $f(t) = f(s)$. From the first it follows
  that $t^{\frown}\tup{\omega} >_{\lex} s^{\frown}\tup{\omega}$, and from
  the second that $c(t^{\frown}\tup{\omega},t) = c(s^{\frown}\tup{\omega},s)$. Summing
  up we get that $t \ge_{\lex} s$, which contradicts that $t \tle s$.

  ($\leftarrow$). As mentioned $(T,<_{\lex})$ is always Countryman, so
  let $c  : T \times T \to \omega$ witness this, i.e.,
  that for all $t,s,t',s' \in T$, if $c(t,s) = c(t',s')$, then
  $t <_{\lex} t' \ra s \le_{\lex} s'$. It is enough to define
  a similar mapping $\tilde{c}$ for $A^{+} \times A^{+}$. For convenience
  our mapping will have range in $\omega^{<\omega}$, which is enough since
  this set is countable.

  Let $f : T \to \omega$ witness that $T$ is special.
  We define $\tilde{c}$ as follows:
  \begin{align*}
    \tilde{c}(t,s) &:= (0,0,f(t),f(s),c(t,s)),\\
    \tilde{c}(t,s^{\frown}\tup{\omega}) &:= (0,1,f(t),f(s),c(t,s)),\\
    \tilde{c}(t^{\frown}\tup{\omega},s) &:= (1,0,f(t),f(s),c(t,s)),\\
    \tilde{c}(t^{\frown}\tup{\omega},s^{\frown}\tup{\omega}) &:= (1,1,f(t),f(s),c(t,s)).
  \end{align*}
  The idea of using the antichains, is that while for $t,t' \in A_{0}$
  $t^{\frown}\tup{\omega} <_{\lex} t'^{\frown}\tup{\omega}$
  does not necessarily follow from $t <_{\lex} t'$, it does follow
  if we also have that $t$ and $t'$ are $T$-incomparable. Thus we have
  that if $f(t) = f(t')$ and $t \neq t'$, then
  \[t <_{\lex} t' \lra
    t^{\frown}\tup{\omega} <_{\lex} t' \lra
    t <_{\lex} t'^{\frown}\tup{\omega} \lra
    t^{\frown}\tup{\omega} <_{\lex} t'^{\frown}\tup{\omega}\]
  This in turn implies that
  if $t,t',s,s' \in A^{+}$, $t \neq t'$, $s \neq s'$ and
  $\tilde{c}(t,s) = \tilde{c}(t',s')$, then
  $t <_{\lex} t' \ra s <_{\lex} s'$. Which finishes the proof.
\end{proof}

By recalling that under $\MA_{\aleph_{1}}$ every tree is special,
we immediately get the following.

\begin{corollary}\label{cor:MA_L_R}
  Assume $\MA_{\aleph_{1}}$. For every $X,Y \seb \omega_{1}$
  we can require $A$ in the conclusion of Theorem~\ref{thm:L_R} to be Countryman.
\end{corollary}

\begin{remark}\label{rmk:zfc_club_cline}
  At the end of Section 4.5 in~\cite{TodorcevicWalks} it is claimed
  that a simple modification of $T^{*}(\rho_{1})$ gives, in $\ZFC$,
  a Countryman line and decomposition $D$
  such that $\La(C,D) \cap \Ra(C,D)$ contains a club.
  The construction as stated does not seem to make sense, and
  in view of Lemma~\ref{lem:countryman_iff_special} and the fact that
  consistently $T^{*}(\rho_{1})$ is not special (see~\cite[2.2.16]{TodorcevicWalks}),
  it is not clear what was intended. However,
  if one is willing to do more work, it is possible
  to construct $A^{+}$ so that is Countryman in $\ZFC$, thus
  removing the need of $\MA_{\aleph_{1}}$ in Corollary~\ref{cor:MA_L_R}.

  We provide a sketch for the interested reader.
  We use the definitions and notations of~\cite[Section 3]{EisworthEtAl2023}.
  Let $U \seb \mathbb{S}$ be $\varrho_{2}$-coherent and full. That such a tree
  exists in $\ZFC$ is pointed out after Lemma 3.15 in that paper. Also $(U,<_{\lex})$
  is Countryman
  by Theorem 3.5 there, and it is easily seen that $U$ is also special.
  Now let $T$ be the result of adding to $U$ the initial
  segments of limit length of elements in $U$. Note then that the $\varrho_{2}$-coherence
  of $U$ implies that if $t \in T \setm U$, then there is a unique
  $n(t) \in \omega$ such that $t^{\frown}\tup{n(t)} \in U$. Let
  $f : T \to U$ be defined by letting $f(t) := t^{\frown}\tup{n(t)}$ if
  $t \in T \setm U$ and $f(t) := t^{\frown}\tup{0}$ otherwise. Then
  $f$ witnesses that $(T,<_{\lex}) \eleq (U,<_{\lex})$, so $(T,<_{\lex})$
  is again Countryman. And $f$ is a tree embedding of $T \setm U$ into
  $U$, so $T$ is again special. Then if one defines $A^{+}$ from this
  $T$ exactly as we did for our tree, one can replicate
  Lemma~\ref{lem:Aplus-dense}, Theorem~\ref{thm:L_R} and Lemma~\ref{lem:countryman_iff_special}
  verbatim.

  Another option would be to construct the $C$-sequence
  as to ensure that $T^{*}(\rho_{1})$ is special. We do
  not know if this is possible in $\ZFC$,
  see~\cite[Question~2.2.18]{TodorcevicWalks}.
\end{remark}

\section{An infinite antichain}\label{sec:infinite-antichain}

In this section we show that already in $\ZFC$ there is a
$\sleq$-antichain of $\aleph_{1}$-dense Aronszajn lines of cardinality
$2^{\aleph_{1}}$, giving a negative answer to Question~\ref{q:sur-wqo}, even
if we restrict to the class of Countryman lines.

\begin{definition}\label{df:approximation}
  Let $A$ be an Aronszajn line and $B \seb A$ be any subset.
  We say that $B$ \emph{approximates} $a \in A$ if $B$
  is cofinal in $\left]-\infty,a\right[$ and coinitial
  in $\left]a,+\infty\right[$.
\end{definition}

Note that if $B$ approximates $a \in A$, and $a$ is not the right
endpoint of $A$, then for every $a' >_{A} a$ there is $b \in B$
such that $a <_{A} b \le_{A} a'$, and if $a' \notin B$, then
the second inequality is strict.

\begin{lemma}\label{lem:non_sur}
  Let $A$ and $X$ be Aronszajn lines. If for some (equivalently any)
  pair of decompositions $D^{A}$ and $D^{X}$ for $A$ and $X$ respectively,
  $\La(A,D^{A}) \setm \La(X,D^{X})$ (or $\Ra(A,D^{A}) \setm \Ra(X,D^{X})$)
  is a stationary subset of $\omega_{1}$,
  then $A \nsgeq X$.
\end{lemma}
\begin{proof}
  Towards a contradiction assume that $\La(A,D^{A}) \setm \La(X,D^{X})$
  is stationary and that $A \sgeq X$.
  Let $f : A \sur X$ be an epimorphism and fix $g : X \to A$
  some inverse of $f$.
  We may assume that $A$ and $X$ are $\omega_{1}$ as sets.
  Now let $E \seb \omega_{1}$ be a club of limit ordinals closed under $f$ and $g$,
  and such that if $\nu \in E$,
  then $\nu$ approximates all its members in the sense of both $<_{A}$ and $<_{X}$.
  This is possible since $A$ and $X$ are Aronszajn
  lines and thus short. Moreover, by Lemma~\ref{lem:dec-at-clubs} we may
  also assume that for $\nu \in E$,
  $D_{\nu} = \nu = E_{\nu}$.

  By assumption there is
  $\nu \in E \cap (\La(A,D^{A}) \setm \La(X,D^{X}))$. This implies that
  some complementary interval of $X \setm \nu$ has no left endpoint. Fix such
  an interval $I$ and let $x \in I$. Then clearly $x \ge \nu$, and
  thus $g(x) \ge \nu$ (otherwise $x = f(g(x)) < \nu$).
  Let $J$ be the complementary interval of $A \setm \nu$ in which $g(x)$ is.
  Since $\nu \in \La(A,D^{A})$, $J$ has a left endpoint, call it $a$.

  We first claim that $f(a) \in I$. If not (including when $f(a) < \nu$), by
  definition there $x'$ such that $f(a) \le_{X} x' <_{X} x$ with $x' < \nu$.
  Since $\nu$ approximates $x'$ and $x \notin \nu$,
  we find $x'' < \nu$ such that $x' <_{X} x'' <_{X} x$. But then
  $a <_{A} g(x'') <_{A} g(x)$, and $g(x'') < \nu$, which contradicts
  that $a$ and $g(x)$ are in the same complementary interval of
  $A \setm \nu$.

  Since $f(a)$ is in $I$, and $I$ does not have a left endpoint, one easily finds
  $x' \in I$ such that $x' <_{X} f(a)$, and thus that
  $g(x') <_{A} a$. Since $a$ is the left endpoint of $J$, there is $a' < \nu$
  such that $g(x') \le_{A} a' <_{A} a$, and thus
  $x' \le_{X} f(a') \le_{X} f(a)$, which implies that $f(a') \in I$. This contradicts
  the fact that $a' < \nu$.
\end{proof}

We now proceed to construct our antichain. This is an adaptation
of the classical way to construct $2^{\aleph_{1}}$ many $\aleph_{1}$-dense
non-isomorphic Aronszajn lines (see~\cite[Theorem~4.8]{Baumgartner1982}),
and it shows that one of the main differences between $\eleq$ and
$\sgeq$ is that epimorphisms are continuous, while embeddings need not to.

\begin{theorem}
  There is an $\sleq$-antichain of
  $\aleph_{1}$-dense Aronszajn lines of size $2^{\aleph_{1}}$.
\end{theorem}
\begin{proof}
  Let $\tup{S_{\xi}
    : \xi < \omega_{1}}$ be a family of disjoint stationary subsets of
  $\omega_{1}$, and for $Z \seb \omega_{1}$ let $L(Z) := \bigcup_{\xi
    \in Z}S_{\xi}$ and $R(Z) := \bigcup_{\xi \in (\omega_{1}\setm
    Z)}S_{\xi}$. Using Theorem~\ref{thm:L_R}, let $A^{Z}$ be an Aronszajn line
  and $D^{Z}$ a decomposition for $A^{Z}$ such that
  $\La(A^{Z},D^{Z}) = L(Z)$ and $\Ra(A^{Z},D^{Z}) = R(Z)$.
  We claim that $\{A^{Z} : Z \seb \omega_{1}\}$ is our desired antichain.

  Let $Z_{0}$ and $Z_{1}$ be distinct subsets of $\omega_{1}$. We may assume
  that $Z_{0}$ is not contained in $Z_{1}$, and thus let
  $\alpha \in Z_{0} \setm Z_{1}$. Then
  $\La(A^{Z_{0}},D^{Z_{0}})\setm \La(A^{Z_{1}},D^{Z_{1}}) =
  L(Z_{0}) \setm L(Z_{1}) \supseteq S_{\alpha}$, which is stationary.
  Thus by Lemma~\ref{lem:non_sur} $A^{Z_{0}} \nsgeq A^{Z_{1}}$.
  Similarly,
  $\Ra(A^{Z_{1}},D^{Z_{1}})\setm \Ra(A^{Z_{0}},D^{Z_{0}}) =
  R(Z_{1}) \setm R(Z_{0}) \supseteq S_{\alpha}$, thus
  by Lemma~\ref{lem:non_sur} $A^{Z_{1}} \nsgeq A^{Z_{0}}$.
\end{proof}

Note that $2^{\aleph_{1}}$ is the largest possible size of an antichain. Also
note that Corollary~\ref{cor:MA_L_R} easily implies the following.

\begin{corollary}\label{cor:cline_antichain_max}
  Assume $\MA_{\aleph_{1}}$. There is an $\sleq$-antichain of
  $\aleph_{1}$-dense Countryman lines of size $2^{\aleph_{1}}$.
\end{corollary}

By Remark~\ref{rmk:zfc_club_cline}, with a little more work this can
also be achieved in $\ZFC$, and even without it, one can also get
an infinite antichain of Countryman lines at the cost of reducing the size.

\begin{corollary}
  There is an $\sleq$-antichain of
  $\aleph_{1}$-dense Countryman lines of size $\aleph_{1}$.
\end{corollary}
\begin{proof}
  Let $\tup{S_{\xi}
    : \xi < \omega_{1}}$ be a family of disjoint stationary subsets of
  $\omega_{1}$ and using Corollary~\ref{cor:countryman_L} let $C^{S_{\xi}}$ and $D^{S_{\xi}}$
  be such that $\La(C^{S_{\xi}},D^{S_{\xi}}) = S_{\xi}$. Then one proves
  that $\{C^{\xi} : \xi < \omega_{1}\}$ is the desired antichain exactly
  as in the proof of the previous theorem.
\end{proof}

\section{An infinite decreasing chain}\label{sec:infinite-chain}

In this section we show that an adaptation of Moore's forcing
in~\cite{Moore2009}, allows us to introduce several epimorphisms
between Countryman lines under $\MA_{\aleph_{1}}$. In particular we
show that under this axiom, there is an infinite $\sle$-decreasing
chain of order type $\omega_{1}$.

Let $\tup{S_{\alpha} : \alpha < \omega_{1}}$ be a family of disjoint
stationary subsets of $\omega_{1}$, and for each $\alpha < \omega_{1}$,
using Corollary~\ref{cor:countryman_L}, let $C^{\alpha}$ be a Countryman
line and $D^{\alpha}$ a decomposition for it such that
$\La(C^{\alpha},D^{\alpha}) = \L(C^{\alpha},D^{\alpha}) = \bigcup_{\xi < \alpha}S_{\xi}$
and $\R(C^{\alpha},D^{\alpha}) = \es$. Lemma~\ref{lem:non_sur} implies that
if $\alpha < \beta$, then $C^{\beta} \nsgeq C^{\alpha}$,
and Corollary~\ref{cor:cline-srs} that $C^{0} \sgeq C^{\alpha}$ for every $\alpha < \omega_{1}$
under $\MA_{\aleph_{1}}$. We claim that this generalizes.

\begin{proposition}
  Assume $\MA_{\aleph_{1}}$. For all $\alpha < \beta$,
  $C^{\alpha} \sge C^{\beta}$.
\end{proposition}

This would indeed give us an infinite $\sle$-decreasing chain of Countryman lines.
We will prove a slightly more general result, given by the following.

\begin{theorem}\label{thm:ma-epi}
  Assume $\MA_{\aleph_{1}}$. Let $A$ and $X$ be $\eleq$-equivalent
  $\aleph_{1}$-dense Countryman lines. If
  for some decomposition $D^{A}$ and $D^{X}$, $\L(A,D^{A}) \seb \La(X,D^X)$ and
  $\R(A,D^{A}) \seb \Ra(X,D^{X})$, then $A \sgeq X$.
\end{theorem}

Observe that if $\alpha < \beta$, then the hypothesis
of the theorem are satisfied letting $A = C^{\alpha}$ and $X = C^{\beta}$. As
a consequence we have another proof of Corollary~\ref{cor:cline-srs}.

\begin{corollary}\label{cor:c-srs2}
  Assume $\MA_{\aleph_{1}}$. Every normal Countryman line is strongly surjective.
\end{corollary}
\begin{proof}
  If $C$ is normal, then in there is a decomposition $D$ for which
  which $\L(C,D) = \R(C,D) = \es$. Then, if $A \eleq C$ is nonempty,
  $A \times C$ is an $\aleph_{1}$-dense Countryman line,
  which is $\eleq$-equivalent to $C$. Thus
 Theorem~\ref{thm:ma-epi} implies that $C \sgeq A \times C$.
\end{proof}

The proof of Theorem~\ref{thm:ma-epi} will be strongly based on Moore's
proof of Theorem~\ref{thm:moore-t1.1}.

\subsection{Moore's forcing}\label{sec:moore-forcing}
In this section we give the necessary details that
we need from Moore's forcing. All definitions and results in this
section are from \cite{Moore2009}.

Let $A$ be an Aronszajn line and assume that $A$ is $\omega_{1}$ as a set.
For $a \in A$ let $\tilde{a} \in 2^{\omega_{1}}$ be the characteristic function
of $\{b \in A : b <_{A} a\}$. One easily sees that
$a \mapsto \tilde{a}$ is an isomorphism between $A$ and a subset of $2^{\omega_{1}}$
ordered lexicographically. For $a \neq b$ in $A$, let
$\Delta_{A}(a,b)$ be the least ordinal $\xi$ such that
$\tilde{a}(\xi) \neq \tilde{b}(\xi)$, and note then that
$a <_{A} b$ iff $\tilde{a}(\Delta_{A}(a,b)) = 0$ and $\tilde{b}(\Delta_{A}(a,b)) = 1$.
More generally, the following are easily seen to be true.

\begin{itemize}\itemsep0em
  \item If $a <_{A} b$, then $\Delta_{A}(a,b)$ is the least ordinal $\xi$ such that
  $a \le_{A} \xi <_{A} b$.
  \item If $a \le_{A} a' <_{A} b' \le_{A} b$,
  then $\Delta_{A}(a,b) \le \Delta_{A}(a',b')$.
  \item For every $a \in A$, $\delta_{A}(a) := \sup\{\Delta_{A}(a,b) + 1 : b\in A \setm \{a\}\}$
  is a countable ordinal.
\end{itemize}

Now let $E \seb \omega_{1}$ be a club of limit ordinals
closed under $\delta_{A}$ and let $T_{A} :=
\{\tilde{a}|_{\xi} : \xi \le \min(E\setm (\delta_{A}(a) + 1))\}$.
Note that $a \mapsto \tilde{a}|_{\delta_{A}(a)}$ maps
$A$ to an antichain of $T_{A}$, so $A$ can be thought of as
the lexicographic ordering of the leafs of $T_{A}$.
Needless to say, $T_{A}$ is an Aronszajn tree.

For the rest of this section fix $A$ and $X$ Countryman lines with
underlying set $\omega_{1}$. The natural forcing for adding
an isomorphism between $A$ and $X$, would be to consider
the set of finite partial functions from $A$ to $X$ that
are increasing. However this is easily seen to collapse
$\omega_{1}$.

\begin{definition}\label{df:moore-forcing}
  For $E \seb \omega_{1}$, the forcing $Q_{E} := Q_{E}(A,X)$ consists of the
  $q$ that are finite partial functions from $A$ to $X$ such that,
  \begin{itemize}\itemsep0em
    \item $q$ is increasing in the sense that if $a <_{A} b$ are in $\dom(q)$,
    then $q(a) <_{A} q(b)$.
    \item For all $\nu \in E$ and $a \in \dom(q)$, $a < \nu \lra q(a) < \nu$.
    \item For all $\nu \in E$ and $a \neq b \in \dom(q)$,
    $\Delta_{A}(a,b) < \nu \lra \Delta_{X}(q(a),q(b)) < \nu$.
  \end{itemize}
\end{definition}

Moore proves the following (see~\cite[Section~3]{Moore2009}).

\begin{lemma}\label{lem:moore-ccc} If $A$ and $X$ are $\eleq$-comparable
Countryman lines. Then there is a club $E$ such that $Q_{E'}$ is ccc
whenever $E'$ is a subclub of $E$.
\end{lemma}

\begin{definition}
  If $S \in H(\omega_{2})$ we say that $E$ is an \emph{elementary club}
  for $S$ if $E$ is contained in
  $\{N \cap \omega_{1} : N \prec H(\omega_{2}), |N| = \aleph_{0}, S \in N\}$.
\end{definition}

It is well known that for any $S \in H(\omega_{2})$ there is
an elementary club for $X$. Moore also proves the following.

\begin{lemma}\label{lem:moore-density}
  If $A$ and $X$ are normal Aronszajn lines, and $E$ is an
  elementary club for $A$ and $X$, then for every $a \in A$ and $x \in
  X$, $\{q \in Q_{E} : a \in \dom(q), x\in \ran(q)\}$ is dense.
\end{lemma}

Combining Theorem~\ref{thm:clines-equiv}, Lemma~\ref{lem:moore-ccc}
and Lemma~\ref{lem:moore-density} one easily obtains a proof of
Theorem~\ref{thm:moore-t1.1}.

We will also need the following, which is stated without proof
in Moore's paper (see the proof of~\cite[Lemma~3.4]{Moore2009}).

\begin{lemma}\label{lem:moore-technical}
  Assume $A$ is Countryman. If
  $\tup{q_{\xi} : \xi < \omega_{1}}$ is a sequence
  of elements of $Q_{\es}$ with pairwise disjoint domains, then
  there is $n < \omega$, an uncountable $\Gamma \seb \omega_{1}$ and $\gamma < \min(\Gamma)$
  such that,
  \begin{enumerate}\itemsep0em
    \item For all $\xi \in \Gamma$, $|\dom(q_{\xi})| = n$.
    \item For all $\xi \in \Gamma$ and $a \neq b \in \dom(q_{\xi})$,
    if $\Delta_{A}(a,b) < \xi$ then $\Delta_{A}(a,b) < \gamma$.
    \item For all $\xi < \eta \in \Gamma$ and $i < n$, if $a$ is
    the $i$-th element in the $<_{A}$-increasing enumeration of $\dom(q_{\xi})$,
    and $b$ the $i$-the element in the increasing enumeration of $\dom(q_{\eta})$,
    then $\gamma < \Delta_{A}(a,b) < \xi$.
  \end{enumerate}
\end{lemma}
\begin{proof}
  One easily finds stationary $\Gamma' \seb \omega_{1}$
  such that (1) holds. We may assume that $\Gamma'$ consists only
  of limit ordinals. Now define $f : \Gamma' \to \omega_{1}$
  by
  \[f(\xi) := \sup(\{0\} \cup \{\Delta_{A}(a,b) + 1 : a\neq b \in \dom(q_{\xi}), \Delta_{A}(a,b) < \xi\}).\]
  By Fodor's lemma, there is stationary $\Gamma '' \seb \Gamma'$
  such that $f$ is constant equal to $\gamma$ on $\Gamma''$. Then
  $\Gamma''' := \Gamma'' \setm (\gamma + 1)$ satisfies (2).

  We now turn to satisfy (3). For $\xi \in \Gamma'''$, let
  $\xi_{0},\dots,\xi_{n-1}$ be the $<_{A}$ increasing enumeration
  of $\dom(q_{\xi})$. Fix $i < n$. Let
  $T = \{t \in T : t \sqsubseteq \tilde{\xi}_{i}
  \text{ for some $\xi \in \Gamma$}\}$, where $\sqsubset$ denotes sequence extension.
  Since $T$ is Aronszajn there is
  $t_{0}$  at a level $\gamma + 1$ with
  extensions at every level. Let $T'$ be the subtree of $T$ of those
  who are comparable with $t_{0}$ and have uncountable extensions, and
  let $H \seb T'$ be an uncountable antichain.
  Then for any $t \in H$ there is $\xi^{(t)} \in \Gamma'''$ such that
  $\tilde{\xi}^{(t)}_{i} \sqsupset t$, and $\xi^{(t)} > \al(t)$.
  Finally let $\Gamma = \{\xi^{t} : t \in H\}$ and
  note that for $\xi < \eta$ in $\Gamma$,
  $\gamma < \Delta_{A}(\xi_{i},\eta_{i}) < \xi$. Repeating this for
  all $i$
  we may assume that $\Gamma$ satisfies (3).
\end{proof}

\subsection{Forcing epimorphisms} Fix $A$, $X$, $D^{A}$ and $D^{X}$ as in the statement of
 Theorem~\ref{thm:ma-epi}, we may assume that $A$ and $X$
are $\omega_{1}$ as sets. We will define a variant on Moore's forcing
$Q_{E}$  to introduce an epimorphism from $A$ onto $X$. As in
Moore's forcing,
the fact that $A$ and $X$ are $\eleq$-comparable Countryman lines
will be used in obtaining the ccc, and the $\aleph_{1}$-density
plus the hypothesis on $\L$ and $\R$ in obtaining enough dense
sets.

Let $\bar{A} = \{(a,b) \in A^{2} : a <_{A} b\}$ and for $\bar{a} \in
\bar{A}$ we write $\bar{a} = (a_{l},a_{r})$ and define $a_{m} :=
\Delta_{A}(a_{l},a_{r})$ so that $a_{l} \le_{A} a_{m} \le_{A} a_{r}$. We
let $<_{b}$ be the partial order on $\bar{A}$ defined by letting $\ba <_{b} \bb$
iff $a_{r} <_{A} b_{l}$, that is, if the interval
${[a_{l},a_{r}]}_{A}$ comes strictly before ${[b_{l},b_{r}]}_{A}$ in the
block order. We always think of elements of $\bar{A}$ as coding
closed intervals. As a set we let $P := P(A,X)$ be the set of finite partial
functions from $\bar{A}$ to $X$ that have domain linearly ordered and
are increasing.  Every $p \in P$ naturally codes a partial epimorphism
from $A$ to $X$ defined by $f_{p}(a) = x$ iff for some $\ba \in
\dom(p)$, $a \in [a_{l},a_{r}]$ and $p(\ba) = x$.  We let $q \le p$
iff $f_{q}$ extends $f_{p}$. This way, a condition $p$ can be extended
either by adding new intervals, or by enlarging an existing one.

The $\aleph_{1}$-density of $A$ and
$X$ implies that for every $p \in P$,
$\{p \in P : a \in \dom(f_{p}), x \in \ran(f_{p})\}$ is dense.
Thus any generic filter for $P$ introduces an epimorphism from $A$
onto $X$.  However it is easily seen that $P$
is never ccc, and in fact collapses $\omega_{1}$.
As in Moore's forcing, it is necessary to refine it
to a ccc forcing.

For a club $E$ and $\alpha < \omega_{1}$, we let
$\nu_{E}(\alpha)$ be the greatest ordinal in $\xi \in E \cup \{0\}$
such that $\xi \le \alpha$. When
the club is clear from the context we drop the subscript and
we let $\alpha^{+}$ stand for the least ordinal $\xi \in E$ such that $\alpha < \xi$.
Also when it is clear from context $\nu(a,b)$ will abbreviate
$\nu(\Delta_{A}(a,b))$, and $\nu(x,y) = \nu(\Delta_{X}(x,y))$.

\begin{definition}\label{df:forcing}
  For $E\seb \omega_{1}$ let $P_{E} := P_{E}(A,X)$ be defined as the set
  of $p \in P$ such that,
  \begin{enumerate}[label = (\roman*)]\itemsep0em
    \item\label{f:i} $\nu(p(\bar{a})) = \nu(a_{l}) = \nu(a_{r}) = \nu(a_{m})$,
    for all $\bar{a} \in \dom(p)$. And thus $\nu(\ba)$ is defined as this common value.
    \item\label{f:ii} $\nu(a_{r},b_{l}) = \nu(p(\bar{a}),p(\bar{b}))$
    for all $\bar{a} <_{b} \bar{b}$ in $\dom(p)$.
    \item\label{f:iii} Let $\bar{a} \in \dom(p)$,
    $I$ be the complementary interval $A \setm \nu(\ba)$
    in which $a_{l}$ is (and then by \ref{f:i} also $a_{r}$)
    and $J$ be the complementary interval of $X \setm \nu(\ba)$
    in which $p(\ba)$ is. Then,
    \begin{itemize}\itemsep0em
      \item If $a_{l}$ is the left endpoint of $I$ then so is $p(\ba)$ of $J$.
      \item If $a_{r}$ is the right endpoint of $I$ then so is $p(\ba)$ of $J$.
    \end{itemize}
    Note than that if $I = [a_{l},a_{r}]$, then $J$ must be singleton.
  \end{enumerate}
\end{definition}

To prove Theorem~\ref{thm:ma-epi} we will find
a club $E$ such that $P_{E}$ is ccc, and such that
for all $a \in A$ and $x \in X$,
$\{p \in P_{E} : a \in \dom(f_{p}), x \in \ran(f_{p})\}$ is dense.
Then an application of $\MA_{\aleph_{1}}$ gives $G$
a generic filter of $P_{E}$, and one sees that
$\bigcup_{p \in G}f_{p}$ is an epimorphism from $A$ onto $X$.
We now turn to the task of finding such a club.

\begin{lemma}\label{lem:ex-club}
  There is a club $E$ such that
  \begin{enumerate}[label = (\arabic*)]\itemsep0em
    \item $E$ makes Moore's forcing $Q_{E}$ ccc.
    \item $E$ is elementary for $A$, $X$, $D^{A}$ and $D^{X}$.
    \item $E$ consists of ordinals closed under $\delta_{A}$ and $\delta_{X}$
    (recall the definition of $\delta_{A}$ from Section~\ref{sec:moore-forcing}).
  \end{enumerate}
\end{lemma}
\begin{proof}
  Using Lemma~\ref{lem:moore-ccc} let $E_{0}$ be a club such that $Q_{E_{0}}$
  is ccc, let $E_{1}$ be a club elementary for all the relevant objects,
  and let $E_{2}$ be any satisfying (3). We claim
  that $E := E_{0} \cap E_{1} \cap E_{2}$ work.

  Conditions (2) and (3) are trivially preserved by taking subclubs, and
  condition (1) is preserved by the choice of $E_{0}$
  (see Lemma~\ref{lem:moore-ccc}).
\end{proof}

We will show that such a club $E$
makes $P_{E}$ ccc and all the needed sets dense. First we prove a
lemma that gives explicit connections between $P_{E}$
and $Q_{E}$. We will often use (a) without mention.

\begin{lemma}\label{lem:P-Q}
  If $E$ is a club of ordinals closed under $\delta_{A}$,
  then,
  \begin{enumerate}[label = (\alph*)]\itemsep0em
    \item For all $a,b \in A$, $\nu(a,b) \le \nu(a),\nu(b)$.
    \item For all $p \in P$ and $\ba \neq \bb \in \dom(p)$,
    if $p$ satisfies~\ref{df:forcing}~\ref{f:i}, then
    $\nu(a_{l},b_{r}) = \nu(a_{m},b_{m}) = \nu(a_{r},b_{l})$.
    \item For all $p \in P$, $p \in P_{E}$ iff
    $q(p) := \{(a_{m},p(\ba)) : \ba \in \dom(p)\}$ is in $Q_{E}$ and
    $p$ satisfies~\ref{df:forcing}~\ref{f:iii}.
  \end{enumerate}
\end{lemma}
\begin{proof}
  (a). Assume $a <_{A} b$. By definition $\Delta_{A}(a,b) \le a$, thus we need to prove
  that $\nu(a,b) \le \nu(b)$. Observe that $b \in {\nu(b)}^{+}$,
  and thus $\delta_{A}(b) < \nu(b)^{+}$. Then clearly $\nu(a,b) \le \nu(b)$.

  (b). Assume $\ba <_{b} \bb$. Then clearly $a_{l} \le_{A} a_{m} \le_{A} a_{r}
  <_{A} b_{l} \le_{A} b_{m} \le_{A} b_{r}$, and thus $\Delta_{A}(a_{l},b_{r}) \le \Delta_{A}(a_{m},b_{m}) \le \Delta_{A}(a_{r},b_{l})$. Therefore we need
  only to prove that $\nu := \nu(a_{l},b_{r}) \ge \nu(a_{r},b_{l}) := \mu$.
  Suppose not, and thus that $\nu < \mu$. Then either
  $\Delta_{A}(a_{l},a_{r}) = \Delta_{A}(a_{l},b_{r})$ or
  $\Delta_{A}(b_{l},b_{r}) = \Delta_{A}(a_{l},b_{r})$. Suppose the second holds,
  the other case is symmetric. Then in particular $\nu(b_{l},b_{r}) = \nu$,
  and since $p \in P_{E}$, $\nu(b_{l}) = \nu(b_{m}) = \nu(b_{l},b_{r}) = \nu$.
  Now using (a) we see that
  $\nu < \mu = \nu(a_{r},b_{l}) \le \nu(b_{l}) = \nu$, which is a contradiction.

  (c). Follows directly from (b) and the definition of $P_{E}$.
\end{proof}

We now prove that the club from Lemma~\ref{lem:ex-club} makes $P_{E}$ ccc.
For this we will only use (1) and (3) of Lemma~\ref{lem:ex-club}.

\begin{theorem}
  If $E$ is a club satisfying (1) and (3) from Lemma~\ref{lem:ex-club},
  then $P_{E}$ is ccc.
\end{theorem}
\begin{proof}
  Suppose towards a contradiction that there is
  an uncountable antichain $H \seb P_{E}$. By going to an uncountable subset of $H$
  we may assume that for all $p \in H$, $|\dom(p)| = n$. Moreover
  we may assume that $H$ is such that $n$ is minimal, in the
  sense that for every other antichain $H'$, the set of $p \in H'$
  with $|\dom(p)| < n$ is countable. From this, a typical
  $\Delta$-system argument shows that the conditions of $H$ must have
  pairwise disjoint domains.
  We will refine $H$ as to satisfy
  that for all $p \neq p' \in H$, if $\ba \in \dom(p)$ and
  $\bb \in \dom(p')$, then $[a_{l},a_{r}] \cap [b_{l},b_{r}] = \es$,
  in other words, $f_{p}$ and $f_{p'}$ have disjoint domains.
  First we show that this suffices to derive our contradiction.

  \begin{claim}
    If $H$ satisfies the above, then $H$ contains two compatible elements.
  \end{claim}
  \begin{proof}
    Since for each $\ba$, $a_{l} \le_{A} a_{m} \le_{A} a_{r}$, the hypothesis
    on $H$ implies that the mapping $p \mapsto q(p)$ is injective when restricted to $H$.
    Thus, using Lemma~\ref{lem:P-Q} (c) and the ccc of $Q_{E}$, we deduce
    that for some $p \neq p'$,
    $q(p)$ and $q(p')$ are compatible, therefore
    $q := q(p) \cup q(p') \in Q_{E}$, since any condition
    witnessing the compatibility of $q(p)$ and $q(p')$ extends $q$.

    Now let $\hat{p} := p \cup p'$. Since $f_{p}$ and $f_{p'}$ have
    disjoint domains, it is clear that $\dom(\hat{p})$ is linearly ordered
    by $<_{b}$. That $\hat{p}$ is increasing follows from the
    fact that $q$ is in $Q_{E}$. That $q \in Q_{E}$ also
    directly implies that $p$ satisfies~\ref{df:forcing}~\ref{f:i},
    and~\ref{df:forcing}~\ref{f:ii}
    follows from the fact that $q \in Q_{E}$ and Lemma~\ref{lem:P-Q} (b).
    Finally~\ref{df:forcing}~\ref{f:iii} follows trivially from the fact that $p$ and $p'$
    satisfy it.
  \end{proof}

  Now we prove that such a refinement is always possible.

  \begin{claim}
    There is an uncountable $H' \seb H$ satisfying that
    for all $p \neq p' \in H'$, $\dom(f_{p}) \cap \dom(f_{p'}) = \es$.
  \end{claim}
  \begin{proof}
    By~\ref{df:forcing}~\ref{f:i} we can find
    $\tup{p_{\xi} : \xi < \omega_{1}}$ contained in $H$
    such that for all $\xi < \eta$,
    $\max\{\nu(\ba) : \ba \in \dom(p_{\xi})\} < \min\{\nu(\ba) : \ba \in \dom(p_{\eta})\}$.
    Now for all $\xi < \omega_{1}$, let $q_{\xi} := q(p_{\xi})$, and
    let $\xi_{0} <_{A} \dots <_{A} \xi_{n-1}$ enumerate $\dom(q_{\xi})$. Observe
    that for all $i < n$, $\xi_{i} = a_{m}$ for some
    $\ba \in \dom(p_{\xi})$. Now we claim that there is an uncountable
    $\Gamma \seb \omega_{1}$ and $\gamma < \min(\Gamma)$ such that,
    \begin{enumerate}[label = (\alph*)]\itemsep0em
      \item $\xi \le \min(\dom(q_{\xi}))$.
      \item For all $i \neq j < n$, $\Delta_{A}(\xi_{i},\xi_{j}) < \xi
      \ra \Delta_{A}(\xi_{i},\xi_{j}) < \gamma$.
      \item For all $i < n$, $\gamma < \Delta_{A}(\xi_{i},\eta_{i}) < \xi$.
    \end{enumerate}
    First note that for all $\xi < \eta$,
    $\ba \in \dom(p_{\xi})$ and $\bb \in \dom(p_{\eta})$,
    $\nu(a_{m}) = \nu(\ba) < \nu(\bb) = \nu(b_{m})$,
    thus $\max(\dom(q_{\xi})) < \min(\dom(q_{\eta}))$. So letting
    $\Gamma$ be a club of ordinals closed under
    $\alpha \mapsto \min\{\xi : \dom(q_{\xi}) \cap \alpha = \es\}$ we easily
    obtain (a).
    (b) and (c) are obtained exactly as in Lemma~\ref{lem:moore-technical}. We
    needed $\Gamma$ to be at least stationary, since (b) requires the use of Fodor's lemma
    over $\Gamma$.

    \begin{figure}[htp]
      \centering
      \caption{Branching pattern}
\begin{tikzpicture}[x=0.75pt,y=0.75pt,yscale=-0.5,xscale=0.5]
  \draw  [line width=0.869] [line join = round][line cap = round] (187,570.5) .. controls (207.17,505.96) and (158.94,438.73) .. (131,384.5) .. controls (116.89,357.1) and (102.54,324.35) .. (98,293.5) .. controls (95.34,275.43) and (100,258.42) .. (100,240.5) ;
  \draw  [line width=0.869] [line join = round][line cap = round] (189,514.5) .. controls (189,516.2) and (188.1,511.2) .. (188,509.5) .. controls (187.38,499.03) and (187.75,495.38) .. (190,485.5) .. controls (194.77,464.5) and (205.09,450.41) .. (220,435.5) .. controls (260.12,395.38) and (279.49,359.09) .. (298,305.5) .. controls (303.85,288.58) and (307.68,274.28) .. (308,256.5) .. controls (308.03,255) and (311.64,194.78) .. (307,185.5) ;
  \draw  [line width=0.869] [line join = round][line cap = round] (99.5,302) .. controls (99.5,295.91) and (82.04,283.52) .. (76.5,276) .. controls (69.58,266.61) and (60.24,256.21) .. (56.5,245) .. controls (53.97,237.42) and (56.5,228.73) .. (56.5,221) ;
  \draw  [line width=0.869] [line join = round][line cap = round] (111,340.5) .. controls (106.04,335.54) and (109.71,326.5) .. (110,319.5) .. controls (110.37,310.66) and (119.14,305.52) .. (123,298.5) .. controls (128.32,288.83) and (132.26,273.54) .. (132,262.5) .. controls (131.63,246.85) and (125,229.86) .. (125,214.5) ;
  \draw  [line width=0.869] [line join = round][line cap = round] (254,395.5) .. controls (254,393.93) and (252.76,394.26) .. (252,393.5) .. controls (247.99,389.49) and (246.3,383.68) .. (244,378.5) .. controls (242.89,376) and (242,364.37) .. (242,362.5) .. controls (242,348.06) and (249.37,334.17) .. (249,319.5) .. controls (248.81,311.77) and (246.81,304.19) .. (246,296.5) .. controls (245,287.03) and (246.28,279.41) .. (241,271.5) ;
  \draw  [line width=0.869] [line join = round][line cap = round] (247.5,332) .. controls (247.5,333.93) and (248.83,326.09) .. (249.5,325) .. controls (251.85,321.18) and (255.05,317.87) .. (258.5,315) .. controls (271.77,303.94) and (277.27,273.77) .. (264.5,261) ;
  \draw  [line width=0.869] [line join = round][line cap = round] (244,346.5) .. controls (243.01,346.5) and (242.77,340.65) .. (239,338.5) .. controls (229.05,332.82) and (223.38,327.9) .. (220,315.5) .. controls (218.38,309.57) and (223.21,305.53) .. (222,299.5) ;
  \draw  [line width=0.4, dash pattern={on 3pt off 1.5pt}]  (110,477) -- (268,476.5) ;
  \draw  [line width=0.4, dash pattern={on 3pt off 1.5pt}]  (80,379.5) -- (298,379.5) ;
  \draw  [line width=0.4, dash pattern={on 3pt off 1.5pt}]  (55.67,355.5) -- (332.67,356.5) ;
  \draw  [line width=0.869] [line join = round][line cap = round] (572,564.5) .. controls (572,556.99) and (575.75,550.93) .. (576,543.5) .. controls (576.64,524.19) and (576.61,504.86) .. (570,486.5) .. controls (557.58,451.99) and (539.09,430.48) .. (519,402.5) .. controls (504.2,381.88) and (494.9,359.09) .. (484,336.5) .. controls (480.44,329.12) and (475.59,322.28) .. (473,314.5) ;
  \draw  [line width=0.869] [line join = round][line cap = round] (576.33,527.67) .. controls (572.67,527.67) and (576.59,517.65) .. (578.33,511.67) .. controls (583.58,493.69) and (597.7,482.03) .. (612.33,471.67) .. controls (633.82,456.45) and (655.77,439.1) .. (669.33,415.67) .. controls (676.05,404.07) and (670.49,395.36) .. (675.33,385.67) ;
  \draw  [line width=0.869] [line join = round][line cap = round] (659,429.33) .. controls (651.05,417.41) and (605.52,409.81) .. (629,386.33) ;
  \draw  [line width=0.869] [line join = round][line cap = round] (667,419.5) .. controls (680.3,419.5) and (687,407.83) .. (687,394.5) .. controls (687,385.89) and (680.13,367.93) .. (689,363.5) ;
  \draw  [line width=0.4, dash pattern={on 3pt off 1.5pt}]  (504.33,490) -- (646.33,490) ;
  \draw  [line width=0.4, dash pattern={on 3pt off 1.5pt}]  (509.67,551.67) -- (643,551.67) ;
  \draw  [line width=0.4, dash pattern={on 3pt off 1.5pt}]  (482.33,449) -- (677.33,449) ;
  \draw  [color={rgb, 255:red, 0; green, 0; blue, 0 }  ,draw opacity=1 ][line width=0.869] [line join = round][line cap = round] (526,411.25) .. controls (526,412.91) and (525.97,410.32) .. (526.5,409.25) .. controls (531.1,400.05) and (535.45,395.49) .. (537.5,386.25) .. controls (540.08,374.65) and (540.39,360.82) .. (537.5,349.25) .. controls (536.58,345.57) and (533.05,329.25) .. (525,329.25) ;
  \draw  [color={rgb, 255:red, 0; green, 0; blue, 0 }  ,draw opacity=1 ][line width=0.869] [line join = round][line cap = round] (540,377.75) .. controls (548.13,371.25) and (558.24,359.7) .. (556.5,348.25) .. controls (555.65,342.69) and (553.97,337.29) .. (553,331.75) .. controls (552.66,329.82) and (552.64,325.25) .. (555.5,325.25) ;
  \draw  [color={rgb, 255:red, 0; green, 0; blue, 0 }  ,draw opacity=1 ][line width=0.869] [line join = round][line cap = round] (565,473.25) .. controls (565,474.59) and (568.65,472.66) .. (573,466.25) .. controls (580.81,454.74) and (579.88,450.46) .. (581.5,433.75) .. controls (582.35,424.98) and (578.06,416) .. (576.5,407.75) .. controls (573.42,391.5) and (573.74,375.14) .. (575,358.75) ;
  \draw (44,192.67) node [anchor=north west][inner sep=0.75pt]   [align=left] {$a_l$};
  \draw (83,204.33) node [anchor=north west][inner sep=0.75pt]   [align=left] {$\xi_i$};
  \draw (112,191.33) node [anchor=north west][inner sep=0.75pt]   [align=left] {$a_r$};
  \draw (210.67,269.33) node [anchor=north west][inner sep=0.75pt]   [align=left] {$b_l$};
  \draw (224,243.67) node [anchor=north west][inner sep=0.75pt]   [align=left] {$\eta_j$};
  \draw (258,230.33) node [anchor=north west][inner sep=0.75pt]   [align=left] {$b_r$};
  \draw (294.67,145) node [anchor=north west][inner sep=0.75pt]   [align=left] {$\xi_j$};
  \draw (275.33,468.33) node [anchor=north west][inner sep=0.75pt]   [align=left] {$\gamma$};
  \draw (300,367.67) node [anchor=north west][inner sep=0.75pt]   [align=left] {$\xi$};
  \draw (336.67,346.33) node [anchor=north west][inner sep=0.75pt]   [align=left] {$\eta$};
  \draw (459.33,290) node [anchor=north west][inner sep=0.75pt]   [align=left] {$a_l$};
  \draw (505.67,285) node [anchor=north west][inner sep=0.75pt]   [align=left] {$\xi_i$};
  \draw (544.83,301) node [anchor=north west][inner sep=0.75pt]   [align=left] {$a_r$};
  \draw (618,346.33) node [anchor=north west][inner sep=0.75pt]   [align=left] {$b_l$};
  \draw (650.33,351) node [anchor=north west][inner sep=0.75pt]   [align=left] {$\eta_j$};
  \draw (676.67,327.67) node [anchor=north west][inner sep=0.75pt]   [align=left] {$b_r$};
  \draw (679.33,437.67) node [anchor=north west][inner sep=0.75pt]   [align=left] {$\eta$};
  \draw (651,480) node [anchor=north west][inner sep=0.75pt]   [align=left] {$\xi$};
  \draw (646,543.67) node [anchor=north west][inner sep=0.75pt]   [align=left] {$\gamma$};
  \draw (567,320.5) node [anchor=north west][inner sep=0.75pt]   [align=left] {$\xi_j$};
\end{tikzpicture}
      \label{figure2}
    \end{figure}

    Now let $\Gamma$ satisfy (a), (b) and (c). We claim that
    $H' := \{p_{\xi} : \xi \in \Gamma\}$ works. For this fix $\xi < \eta$ in $\Gamma$,
    and let $\ba \in \dom(p_{\xi})$ and $\bb \in \dom(p_{\eta})$.
    We prove that $[a_{l},a_{r}] \cap [b_{l},b_{r}] = \es$.
    By definition of $q_{\xi}$,
    there are $i,j < n$ such that $\xi_{i} = a_{m}$
    and $\eta_{j} = b_{m}$. We claim that
    $\Delta_{A}(\xi_{i},\eta_{j}) < \xi$.

    If $i = j$ this follows directly from (c). So assume $i \neq j$. There
    are two cases depending on whether $\Delta_{A}(\xi_{i},\xi_{j}) < \xi$ or
    not.
    If $\Delta_{A}(\xi_{i},\xi_{j}) < \xi$ by (b) we have that
    $\Delta_{A}(\xi_{i},\xi_{j}) < \gamma$, and
    by (c) that $\gamma < \Delta_{A}(\xi_{j},\eta_{j})$, then
    $\Delta_{A}(\xi_{i},\eta_{j}) < \gamma < \xi$ follows.
    If $\Delta_{A}(\xi_{i},\xi_{j}) \ge \xi$, then by (c)
    we have $\Delta_{A}(\xi_{j},\eta_{j}) < \xi$, it follows that
    $\Delta_{A}(\xi_{i},\eta_{j}) < \xi$. See Figure~\ref{figure2}.

    Now suppose that $[a_{l},a_{r}] \cap [b_{l},b_{r}] \neq \es$
    and assume that $\xi_{i} <_{A} \eta_{j}$, the other case is analogous.
    Then, $c := \Delta_{A}(\xi_{i},\eta_{j}) \in {[\xi_{i},\xi_{j}]}_{A}
    \seb [a_{l},a_{r}] \cup [b_{l},b_{r}]$, so $c \in [a_{l},a_{r}]$
    or $c \in [b_{l},b_{r}]$ (or both). Either way we arrive at a contradiction
    since $c = \Delta_{A}(\xi_{i},\eta_{j}) < \xi$, and (a) implies
    that $\Delta_{A}(a_{l},a_{r}) = a_{m} \ge \xi$ and
    $\Delta_{A}(b_{l},b_{r}) = b_{m} \ge \eta > \xi$.
  \end{proof}
\end{proof}

We dedicate the rest of this section to show that the $E$ from
 Lemma~\ref{lem:ex-club} makes all the needed sets dense. For this
we only use properties (2) and (3) (in fact (3) follows from (2)).

Say that $\ba,
\bb \in \dom(p)$ are \emph{neighbors} if there is no $\bar{c} \in
\dom(p)$ strictly between them. When checking that new conditions are
in $P_{E}$ it will be useful to have the following.

\begin{lemma}\label{lem:neighbors}
  Let $p \in P$ satisfy~\ref{df:forcing}~\ref{f:i}
  and~\ref{df:forcing}~\ref{f:iii}. If
  for all neighbors $\ba <_{b} \bb$ in $\dom(p)$,
  $\nu(a_{r},b_{l}) = \nu(p(\ba),p(\bb))$, then $p \in P_{E}$.
\end{lemma}
\begin{proof}
  It is enough to prove that if $\ba <_{b} \bb <_{b} \bc$ are in
  $\dom(p)$ and $\ba$,$\bc$ are neighbors of $\bb$,
  then $\nu(a_{r},c_{l}) = \nu(p(\ba),p(\bc))$. By Lemma~\ref{lem:P-Q}
  (b), it is enough to prove that $\nu(a_{m},c_{m}) = \nu(p(\ba),p(\bb))$.

  Since $p(\ba) <_{X} p(\bb) <_{X} p(\bc)$ either
  $\Delta(p(\ba),p(\bc)) = \Delta(p(\ba),p(\bb))$ or
  $\Delta(p(\ba),p(\bc)) = \Delta(p(\bb),p(\bc))$. Assume
  the second, the other case is symmetric. So $\nu :=
  \nu(p(\ba),p(\bc)) = \nu(p(\bb),p(\bc))$. Since $p$
  satisfies~\ref{df:forcing}~\ref{f:ii}
  for neighbors, $\nu = \nu(b_{m},c_{m})$ and
  $\mu := \nu(a_{m},b_{m}) = \nu(p(\ba),p(\bb))$. Clearly $\nu \le \mu$,
  and if equality holds we are done. So assume
  $\nu < \mu$. This implies that $\Delta(b_{m},c_{m}) = \Delta(a_{m},c_{m})$,
  and then $\nu(a_{m},c_{m}) = \nu(b_{m},c_{m}) = \nu = \nu(p(\bb),p(\bc)) =
  \nu(p(\ba),p(\bc))$.
\end{proof}

The following summarizes some properties that follow from the
elementarity of $E$ that will be frequently used without mention. Their
proofs are standard, we give a proof of the last one as an
example.

\begin{lemma}\label{lem:elementarity}
  \begin{enumerate}[label=(\alph*)]\itemsep0em
    \item For all $\nu \in E$, $D_{\nu} = \nu$.
    \item Every $\nu \in E$ approximates all its elements
    (recall Definition~\ref{df:approximation}).
    \item For all $a \neq b$ and $\nu \in E$, $\Delta(a,b) \ge \nu$ iff $a,b \ge \nu$
    and they are in the same complementary interval of $A \setm \nu$.
    \item If $\nu \in E$ and $a$ is an endpoint of some complementary interval of
    $A \setm \nu$, then $\nu(a) = \nu$.
    \item If $\nu \in E \cup \{0\}$,
    then $\left[\nu,\nu^{+}\right[$ is cofinal and coinitial in every complementary
    interval of $A \setm \nu$.
  \end{enumerate}
\end{lemma}
\begin{proof}
  We prove the last one, the others are similar. Let $N \prec H(\omega_{2})$
  witness that $\nu^{+}$ is elementary for $A$. Observe
  that $H(\omega_{2})$ models that there is a countable
  $X \seb A$ such that $X$ is cofinal and coinitial in every complementary
  interval of $A \setm \nu$, this is because there are countably many such intervals.
  Then, there must be such a set $X$ inside $N$, and since $X$ is countable,
  $X \seb N$. Since $N \cap \omega_{1} = \nu^{+}$,
  this implies that $X \seb \nu^{+}$.
  That $X \cap \nu = \es$ is by definition.
\end{proof}

\begin{lemma}\label{lem:tech-lemma}
  Let $a,b \in A$ and $\mu \in E \cup \{0\}$ be such that
  $a <_{A} b$, and $\nu(a,b) \le \mu$. Then there are
  $a',b'\in A$ such that:
  \begin{enumerate}[label=(\alph*)]\itemsep0em
    \item $\nu(a') = \nu(b') = \nu(a',b') = \mu$.
    \item $a <_{A} a' <_{A} b' <_{A} b$.
    \item $a'$ and $b'$ are not endpoints of their
    (by (a) they are in the same one) complementary interval
    of $A \setm \mu$.
  \end{enumerate}
\end{lemma}
\begin{proof}
  We first find $c', c'' < \mu^{+}$ such that
  $a <_{A} c' <_{A} c'' <_{A} b$.
  Let $c := \Delta(a,b)$, and observe that $a \le_{A} c <_{A} b$ and
  $\nu(c) = \nu(a,b)$ hold.
  Using the $\aleph_{1}$-density of $A$, and Lemma~\ref{lem:elementarity} (b) for $\nu(a,b)^{+}$
  together with the fact that $\nu(c) < \nu(a,b)^{+}$, approximating $c$ from the right
  one finds
  $c',c'' \in A$ such that $c <_{A} c' <_{A} c'' <_{A} b$ and
  $\nu(c') = \nu(c'') = \nu(a,b) \le \mu$, thus $c',c'' < \mu^{+}$.

  Now by $\aleph_{1}$-density of $A$, there must be an open interval $I$ of $A$
  such that $I \seb [c',c'']_{A}$ and
  $I \cap \mu = \es$, i.e., it is contained in a complementary interval
  of $A \setm \mu$, then any $a',b' \in I$ satisfy (b), (c) and
  $\mu \le \nu(a',b') \le \nu(a'),\nu(b')$. Applying the above inside a
  model witnessing the elementarity of $\mu^{+}$, and using the
  fact that $c',c'' < \mu^{+}$, one also achieves
  that $\nu(a'),\nu(b') < \mu^{+}$.
\end{proof}

We now proceed to prove the relevant densities. As explained before
 Lemma~\ref{lem:ex-club}, this together with the ccc finishes the proof of
 Theorem~\ref{thm:ma-epi}. We will prove the first of these
lemmas giving all the necessary details and tricks, in the second
we give less details but the tricks are the same.

\begin{lemma}
   For every $x \in X$,
  $\{p \in P_{E} : x \in \ran(f_{p})\}$ is dense.
\end{lemma}
\begin{proof}
  Assume $p \in P_{E}$ and $z \notin \ran(p)$.
  Let $x <_{X} y$ be neighbors to $z$ in $\ran(p)$.
  The case when $x$ and/or $y$ do not exist is similar.
  Because the arguments are symmetric, we may assume
  that $\Delta(x,z) < \Delta(z,y)$. Let $\nu := \nu(z,y)$,
   $\bar{a} := p^{-1}(x)$ and $\bar{b} := p^{-1}(y)$.
  Note that by Lemma~\ref{lem:P-Q} (a), $\nu \le \nu(y)$,
  and by definition of $P_{E}$,
  $\nu(y) = \nu(b_{l})$. Observe also that
  $\Delta(x,z) = \Delta(x,y)$.

  First we claim that $b_{l}$ is not the left endpoint of its complementary interval
  of $A \setm \nu$. If $\nu < \nu(y)$ then this follows from Lemma~\ref{lem:elementarity} (d)
  and the fact that $\nu(y) = \nu(b_{l})$.
  If $\nu = \nu(y)$, then it follows
  from~\ref{df:forcing}~\ref{f:iii} and
  the fact that $y$ is not the left endpoint of its complementary interval
  of $X \setm \nu(\bb)$, since $z <_{X} y$ and $\Delta(z,y) = \nu = \nu(y)$.

  Now we claim that there is $c'$ such that $\nu(c') = \nu$, $a_{l} <_{A} c' <_{A} b_{l}$,
  and such that $c'$ is in the same complementary interval of $A \setm \nu$ in
  which $b_{l}$ is. There are two cases, if $\nu(a_{r},b_{l}) < \nu$ (that is if they are
  not in the same complementary interval of $A \setm \nu$)
  then use Lemma~\ref{lem:elementarity} (e) and the fact that $b_{l}$ is not the
  left endpoint of its complementary interval in $A \setm \nu$. If
  $\nu(a_{r},b_{l}) = \nu$, then $\Delta(a_{r},b_{l}) < \nu^{+}$
  and $a_{l} \le_{A} \Delta(a_{r},b_{l}) <_{A} b_{r}$, and then one can
  use Lemma~\ref{lem:elementarity} (b) and $\aleph_{1}$-density to
  approximate $\Delta(a_{r},b_{l})$ from the right, thus finding $c'$.

  Repeating the above argument, and using the $\aleph_{1}$-density  of $A$,
  we find $c''$ such that $\nu(c'') = \nu$ and
  $c' <_{A} c'' <_{A} b_{l}$. Now note that $\nu \le \nu(z)$,
  thus using Lemma~\ref{lem:tech-lemma} one finds $c_{l}$ and $c_{r}$ such that
  $\nu(c_{l}) = \nu(c_{r}) = \nu(z)$, $c' <_{A} c_{l} <_{A} c_{r} <_{A} c''$,
  $c_{l}$ and $c_{r}$ are in the same complementary interval of $A \setm \nu(z)$
  (that is $\Delta(c_{l},c_{r}) \ge \nu(z)$) and are not endpoints of it.
  We claim that $p' := p \cup \{((c_{l},c_{r}),z)\}$ is in $P_{E}$.

  By construction $[c_{l},c_{r}]$ is disjoint from $\dom(f_{p})$,
  and $p'$ is increasing. Condition~\ref{df:forcing}~\ref{f:i}
  follows from the
  fact that $\nu(c_{l})= \nu(c_{r}) = \nu(z)$ and
  $\nu(c_{l},c_{r}) = \nu(z)$
  because then $c_{l}$ and $c_{r}$ are in the same complementary
  interval of $A \setm \nu(z)$. \Ref{df:forcing}~\ref{f:iii}
  follows because we chose
  $c_{l}$ and $c_{r}$ to not be endpoints of $I$. Using Lemma~\ref{lem:neighbors}
  we see that~\ref{df:forcing}~\ref{f:ii} holds because by construction
  $\Delta(a_{r},c_{l}) = \Delta(a_{r},b_{l}) = \Delta(x,z)$,
  and $\nu(c_{r},b_{l}) = \nu = \nu(z,y)$.
\end{proof}

\begin{lemma}
  For every $a \in A$,
  $\{p \in P_{E} : a \in \dom(f_{p})\}$ is dense.
\end{lemma}
\begin{proof}
  Take $p \in P_{E}$ and $c \notin \dom(f_{p})$.
  Then let $\bar{a} < \bar{b}$ neighbors to $c$ in $\dom(p)$. The
  case when there is no such $\bar{a}$ and/or $\bar{b}$ is similar.
  Because the arguments are symmetric, we may assume
  that $\Delta(a_{r},c) < \Delta(c,b_{l})$. Let $\nu := \nu(c,b_{l})$,
  $x := p(\bar{a})$ and and $y := p(\bar{y})$.
  Then
  \[\nu(x,y) = \nu(a_{r},b_{l}) \le \nu \le \nu(b_{l})
    = \nu(y)\]

  \underline{Case 1}: $\nu = \nu(a_{r},b_{l})$. Note that then
  $\nu = \nu(x,y)$ and that in this case $c$ is not
  an endpoint of its complementary interval of $A \setm \nu$.
  Using elementarity
  and $\aleph_{1}$-density one finds $a_{r}<_{A}c_{l} <_{A} c <_{A} c_{r} <_{A} b_{l}$
  such that $\nu(c_{l})=\nu(c_{r}) = \nu$, and $x <_{X} z <_{X} y$ such that
  $\nu(z) = \nu$ (and thus also $\Delta(x,z) = \Delta(x,y) = \nu$).
  Then $p \cup \{((c_{l},c_{r}),z)\}$ is in $P_E$.

  \underline{Case 2}: $\nu(a_{r},b_{l}) < \nu$ and $y$ is
  the left endpoint of its complementary interval in $X \setm \nu$.
  Note that in this case it must be that $\nu = \nu(y) = \nu(b_{l})$.
  Using elementarity one finds $c_{l} \le_{A} c$ in the complementary interval
  of $A \setm \nu$ of $b_{l}$ such that $\nu(c_{l}) = \nu$. We claim that
  $p' := (p \setm \{((b_{l},b_{r}),y)\}) \cup \{((c_{l},b_{r}),y)\}$ is in $P_E$.
  \Ref{df:forcing}~\ref{f:iii} is trivially satisfied since $y$ is the left
  endpoint of its complementary interval of $A \setm \nu$, and noting that
  $\nu = \nu(c_{l}) = \nu(c_{l},b_{r}) = \nu(b_{l},b_{r})$, one sees
  that~\ref{df:forcing}~\ref{f:i} is satisfied, as well as~\ref{df:forcing}~\ref{f:ii}.

  \underline{Case 3}: $\nu(a_{r},b_{l}) < \nu$ and $y$ is
  not the left endpoint of its complementary interval of $X \setm \nu$.

  \underline{Subcase 3.1}: $c$ is the left endpoint of its complementary interval in $A \setm \nu$. It is in this case where we need to use the hypothesis on
  $\L$ and $\R$. Note that in this case we have $\nu(c) = \nu$.
  By the hypotheses on $A$ and $X$, and noting
  that by Lemma~\ref{lem:elementarity} (a) $A \setm D_{\nu}^{A} = A\setm\nu$, we know
  that every complementary interval
  of $X \setm D_{\nu}^{X} = X \setm \nu$ has a left endpoint. Let $z$ be the left endpoint of
  the complementary interval of $X \setm \nu$ in which $y$ is, and note that then
  $x <_{X} z <_{X} y$ and $\nu(z) = \nu$ by Lemma~\ref{lem:elementarity} (d).
  As in Case 1 one finds
  $c_{r}$ such $\nu(c_{r}) = \nu$, $c_{l} <_{A} c_{r} <_{A} b_{l}$ and
  $\nu(c_{l},c_{r}) = \nu$. Then  $p \cup \{((c,c_{r}),z)\}$ is in $P_E$.

  \underline{Subcase 3.2}: $c$ is not the left endpoint of its complementary interval in $A \setm \nu$.
  Using elementarity and $\aleph_{1}$-density let $z$ be in the same complementary
  interval of $X\setm \nu$ of $y$, and such that $\nu(z) = \nu$ and $z <_{X} y$.
  This can be done since we know that $y$ is not the left endpoint. As in Case
  1 one finds $c_{l}$ and $c_{r}$ such that $c_{l} <_{A} c <_{A} c_{r}$,
  $\nu(c_{l}) = \nu(c_{r}) = \nu$ and $\nu(c_{l},c_{r}) = \nu$ and thus
  $p \cup \{((c_{l},c_{r}),z)\} \in P_{E}$.
\end{proof}

\section{A two element basis for the Aronszajn lines}\label{sec:basis}

Although we have shown that the analogy between countable linear orders
and Aronszajn lines does not extend perfectly to $\sleq$, we do have a positive
result.

Recall that in~\cite{CamerloEtAl2015} it is shown that $\omega$, $\omega + 1$ and
$\omega^{\star}$ and $1 + \omega^{\star}$ form a $\sleq$-basis for the countable
linear orders. We do have an extension of this to the class of Countryman lines.
Using the results of Section~\ref{sec:decompositions}, fix an $\aleph_{1}$-dense
Countryman line $C$ with decomposition $D$
such that $\La(C,D) = \Ra(C,D) = \omega_{1}$.

\begin{lemma}\label{lem:countryman_basis}
  Assume $\MA_{\aleph_{1}}$.
  $1 + C + 1$ and $1 + C^{\star} + 1$ form a $\sleq$-basis for the Countryman lines.
\end{lemma}
\begin{proof}
  It is enough to prove that if $A \eleq C$ is uncountable, then $1 + C + 1 \sleq A$.
  First observe that $A \sgeq 1 + B + 1$ for some $\aleph_{1}$-dense
  Aronszajn line $B$. This can be proven by taking the usual
  quotient on $A$ defined by $a \sim b$ if there are countably many points between
  them. Then Theorem~\ref{thm:ma-epi} implies that $B \sgeq C$ which finishes the proof.
\end{proof}

Under $\PFA$ we can say more. Recall that Moore
proved~\cite[Theorem~1.3]{Moore2009} that under $\PFA$
any non-Countryman Aronszajn line contains a copy of $L$ and
$L^{*}$ for any given Countryman line $L$.

\begin{theorem}\label{thm:alines-basis}
  Assume $\PFA$. $1 + C + 1$ and $1 + C^{\star} + 1$ form a $\sleq$-basis for the
  Aronszajn lines.
\end{theorem}
\begin{proof}
  Taking into account Lemma~\ref{lem:countryman_basis}
  it is enough to prove that if $A$ is any Aronszajn line, then
  $A$ surjects onto some Countryman line.

  By Theorem~\ref{thm:univ-aline-strong} there is $B \seb A$ that is an
  $\aleph_{1}$-dense Countryman line with the induced order.
  For $a <_{A} b$, define $a \sim b$ if $[a,b]_{A} \cap B = \es$.
  We claim that $A / \sim$ is Countryman, this finishes the proof
  since clearly $A \sgeq A /\sim$.

  Suppose not, then by~\cite[Theorem~1.3]{Moore2009}
  $A / \sim$ contains a copy of $B^{\star}$, and thus
  there is $f : B^{\star} \to A$ such that for every $x <_{B^{\star}} y$,
  $f(x) <_{A} f(y)$ and $f(a) \nsim f(b)$. Since $B^{\star}$ is $\aleph_{1}$-dense
  and not Suslin (use $\PFA$ or \cite[Fact~3.9]{Martinez}),
  there is
  $\{{[x_{\xi},y_{\xi}]}_{B^{\star}} : \xi < \omega_{1}\}$ a collection
  of disjoint closed intervals of $B^{\star}$. Now
  for each $\xi < \omega_{1}$, $f(x_{\xi}) <_{A} g(y_{\xi})$ and
  $f(x_{\xi}) \nsim g(y_{\xi})$, and thus there is some
  $a_{\xi} \in B \cap {[f(x_{\xi}),f(y_{\xi})]}_{A}$.
  Then $x_{\xi} \mapsto a_{\xi}$ is an embedding from an uncountable
  suborder of $B^{\star}$ into $B$, which contradicts Lemma~\ref{lem:cline-fund}.
\end{proof}

\subsection{On a basis for all uncountable linear orders}\label{sec:basis-all}

Theorem~\ref{thm:alines-basis} suggests the following question. Under $\PFA$,
is there a finite $\sleq$-basis for the uncountable linear orders?
It is easily seen that if $L$ is not short,
then $L$ surjects onto $\omega_{1}$, $\omega_{1} + 1$, $\omega_{1}^{\star}$
or $1 + \omega_{1}^{\star}$. However we will see that no matter what
set-theoretic hypotheses are in use, this cannot be extended to all uncountable
linear orders. This will be done by proving (in
$\ZFC$) the following theorem.

\begin{theorem}\label{thm:real-basis}
  Any $\sleq$-basis for the uncountable real orders has at
  least $\cc^{+}$ elements.
\end{theorem}

Let $\sF$ be the family of epimorphisms between subsets of reals. That
is, $f \in \sF$ iff $\dom(f),\ran(f) \seb \mathbb{R}$, and for all $x
< y$ in $\dom(f)$, $f(x) \le f(y)$.  For $f \in \sF$ let $f
\!\uparrow$ denote the set of $g \in \sF$ such that $f \seb g$. Define
$\bar{f}$ as the set of $(x,y) \in \mathbb{R}^{2}$ such that for all
$g \in f\!\uparrow$ with $x \in \dom(g)$, $g(x) = y$. Morally, we want
$\bar{f}$ to be the largest monotone function whose values are
uniquely determined by $f$. In practice, it will be convenient to have the
following characterization.

\begin{lemma}\label{lem:hatf-char}
  For every $f \in \sF$ and $(x,y) \in \mathbb{R}^{2}$, the following are equivalent:
  \begin{enumerate}[label=(\alph*)]\itemsep0em
    \item $(x,y) \in \bar{f}$.
    \item $\sup(f (\dom(f) \cap \left]-\infty,x\right])) =
    y = \inf(f (\dom(f) \cap \left[x,+\infty\right[))$.
  \end{enumerate}
\end{lemma}
\begin{proof}
  Assume that
  $(b)$ fails. Then there
  $y_{0} < y_{1}$ such that $y \in [y_{0},y_{1}]$,
  $f (\dom(f) \cap \left]-\infty,x\right]) \seb \left]-\infty, y_{0}\right]$ and
  $f (\dom(f) \cap \left[x,+\infty\right]) \seb \left[y_{1}, +\infty\right[$.
  Note then that $x \notin \dom(f)$.
  Then one easily sees that both $f \cup \{(x,y_{0})\}$ and
  $f \cup \{(x,y_{1})\}$ are in $f\!\uparrow$, which implies
  that $(x,y) \notin \bar{f}$, so $(a)$ fails.

  Assume $(b)$. If $x \in \dom(f)$ then clearly $y = f(x)$ and thus
  $(a)$ holds. Assume
  that $x \notin \dom(f)$. Then it
  should be clear that $f \cup \{(x,y)\} \in f \!\uparrow$.
  Now suppose that $(a)$ fails, and thus take $g \in f\!\uparrow$ such that
  $y' := g(x) \neq y$. We may assume that $y < y'$. By $(b)$ there is
  $x_{0} \in \dom(f)$ such that $x < x_{0}$ and $y \le f(x_{0}) < y'$. But since $g$ extend $f$,
  we have that $x < x_{0}$ and $g(x_{0}) < g(x) = y'$, which contradicts that
  $g$, being in $f\!\uparrow$, is monotone.
\end{proof}

Note that this implies that for every $f \in \sF$,
$\bar{f}$ is indeed a monotone function extending $f$.

The following three lemmas are essentially due to Baumgartner (see
\cite{Baumgartner1982}). We give a proof of the following because
Baumgartner only deals with monotone injective functions, and
passing to all monotone functions requires a small tweak.

\begin{lemma}
  For every $f \in \sF$ there is a countable $g \seb f$ such that
  $\bar{g} \in f\!\uparrow$. Moreover, if $f$ is injective then so is $\bar{g}$.
\end{lemma}
\begin{proof}
  Let $A$ be a countable subset of $\dom(f)$
  such that every $x \in \dom(f)$ is either in
  $A$, or is a limit point (in $\mathbb{R}$) of both
  $\{a \in A : a < x\}$ and $\{a \in A : x < a\}$. Let $B \seb \ran(f)$ be
  defined similarly. Now using the ccc of $\mathbb{R}$ and
  the fact that $f$ is monotone, one sees that there are only
  countably many $y \in \ran(f)$ such that $f^{-1}(y)$ has
  more than one element. Therefore, by enlarging $A$ and $B$ we may assume
  $B$ contains every such $y$, and $A$ contains the endpoints
  of $f^{-1}(y)$ if it has any. Finally, enlarging $A$ and $B$ again
  we may assume that $f(A) = B$. We claim that
  $g := f|_{A}$ works.

  Note that the previous lemma implies
  that $\bar{g} \in \sF$, and that $\bar{g}$ is injective whenever
  $g$ is. To prove that $f \seb \bar{g}$ is enough to show that
  if $x \in \dom(f) \setm A$, and $y = f(x)$, then
  $(x,y) \in \bar{g}$.

  Case 1: $|f^{-1}(y)| > 1$. Observe that $x$ cannot be the left or right
  endpoint of $f^{-1}(y)$, because by construction this would imply that
  $x \in A$. Thus one finds $x_{0} < x < x_{1}$ in $A \cap f^{-1}(y)$.
  So clearly $(b)$ of the previous lemma holds for $g$ and $(x,y)$, and then
  $(x,y) \in \bar{g}$.

  Case 2: $f^{-1}(y) = \{x\}$. Since $x \notin A$, and $f(A) = B$,
  we conclude that $y \notin B$. Then by definition of $B$,
  we have that $\sup(B \cap \left]-\infty,y\right[) =
  \inf(B \cap \left]y,+\infty\right[)$, from which it follows that both
  are equal to $y$. Now, using the fact that $f(A) = B$ again, we
  have that $g$ and $(x,y)$ satisfy $(b)$ of the previous lemma, and
  thus $(x,y) \in g$.
\end{proof}

Since there are only $\cc$ many countable elements in $\sF$, the previous
implies that there is a sequence $\tup{f_{\alpha} : \alpha < \cc}$ of elements
in $\sF$ such that for every $f \in \sF$, there is $\alpha$ such that
$f \seb f_{\alpha}$. Moreover, if $f$ is injective then the $\alpha$ can be
chosen so that $f_{\alpha}$ is also injective.

We now construct by recursion
a sequence $\tup{A_{\alpha} : \alpha < \cc}$ of pairwise disjoint
subsets of $\mathbb{R}$. The construction closely follows
that of~\cite{Baumgartner1982}, we only generalize it as to allow
adding more than one (but not too many) elements at each step:
At step $\alpha$, assume that $A_{\xi}$ has been defined for every $\xi < \alpha$,
that they are pairwise disjoint, and that for each $\xi < \alpha$,
$1 \le |A_{\xi}| \le \max(\aleph_{0},|\xi|)$. Note then
that $U_{\alpha} := \bigcup_{\xi < \alpha}A_{\alpha}$ has cardinality less
than $\cc$. Let $Z_{\alpha}$ be the closure of $U_{\alpha}$ under $\mathrm{id}$,
and every $f_{\xi}$ and $f_{\xi}^{-1}$ (when $f_{\xi}$ is injective)
for $\xi < \alpha$. Then $U_{\alpha} \seb Z_{\alpha}$ and $|Z_{\alpha}| < \cc$.
Now we let $A_{\alpha}$ be any subset of $\mathbb{R}$ disjoint from $Z_{\alpha}$,
and such that $1 \le |A_{\alpha}| \le \max(\aleph_{0},|\alpha|)$. We later
specify a particular way to chose the $A_{\alpha}$ as to achieve a particular
property, but for the following two lemmas any choice satisfying the above works.

For $X \seb \cc$, let $A(X) := \bigcup_{\alpha \in X}A_{\alpha}$. The following
lemma (and its proof) is a straightforward modifications of Lemma 2.3
in~\cite{Baumgartner1982}.

\begin{lemma}
  Suppose that $Y \seb \cc$, and that $B \seb \mathbb{R}$ is such that
  $\{\alpha < \cc : B \cap A_{\alpha} \neq \es\} \setm Y$ is cofinal.
  Then $B \neleq A(Y)$.
\end{lemma}

In~\cite{Baumgartner1982} (Theorem 1.2 (a)), it is proven that there is
a family $S$ of size $\cc^{+}$ of subsets of $\cc$ such that for every
$X \neq Y$ in $S$, $|X| = |Y| = \cc$ and $\sup(X\cap Y) < \cc$. For
the rest of this section we fix such a family. The following
is Theorem 2.4 (b) of \cite{Baumgartner1982}. We give the proof because
is not given there.

\begin{lemma}
  Let $X \neq Y$ be in $S$. If $B \eleq A(X), A(Y)$, then $|B| < \cc$.
\end{lemma}
\begin{proof}
  Suppose that $|B| = \cc$ and that $B \eleq A(X)$, we may assume in fact that
  $B \seb A(X)$. Since for each $\alpha$,
  $|A_{\alpha}| < \max(\aleph_{0},|\alpha|)$, we conclude that
  $X' := \{\alpha < \cc : B \cap A_{\alpha} \neq \es\}$ is cofinal. Since
  $\sup(X \cap Y) < \cc$, $X' \setm Y$ is also cofinal, so by
  the previous lemma, $B \neleq A(Y)$.
\end{proof}

Note that this already implies Theorem~\ref{thm:real-basis} under $\CH$,
since any $\sleq$-basis is an $\eleq$-basis. What we show now, is that
for a particular way of choosing the $A_{\alpha}$ we can achieve the
following property: for $X \seb \cc$ and $B \seb \mathbb{R}$,
if $\aleph_{0} < |B| < \cc = |X|$, then $A(X) \nsgeq B$.
Together with the previous lemma, this immediately
yields a proof of Theorem~\ref{thm:real-basis}.

Let $L$ be a linear order. Consider a pair
of subsets $(X,Y)$ such that $X$ is exactly the
set of lower bounds of $Y$, and $Y$ is the set of
upper bounds of $X$. We furthermore ask $X$ and
$Y$ to be nonempty. Such a pair is called a \emph{gap}
of $L$ if $X \cap Y = \es$. Note that in this case
$X$ has no right endpoint, and $Y$ has no left endpoint.
Let $G(L)$ be the set of gaps of $L$. This is naturally
linearly ordered by letting $(X,Y) \le (X',Y')$ iff
$X \seb X'$. By taking preimages the following
is easily proved (or see \cite[Proposition~5.3]{CamerloEtAl2019}).

\begin{proposition}
  If $B \sleq A$, then $G(B) \eleq G(A)$.
\end{proposition}

\begin{lemma}\label{prop:gaps}
  Let $A,B \seb \mathbb{R}$. If  $A \sgeq B$ and $|\mathbb{R}\setm A| < \cc$,
  then $B$ is countable or $|B| = \cc$.
\end{lemma}
\begin{proof}
  Let $f : A \sur B$ witness that $A \sgeq B$, so $f \in \sF$. Because
  $|\mathbb{R}\setm A| < \cc$, $\mathbb{R}= A \cup E$ where
  $|E| < \cc$. Since $\mathbb{R}$ is complete (i.e., has no
  gaps), $|G(A)| \le |E| < \cc$. Then the previous proposition yields
  $|G(B)| < \cc$ also. Now suppose towards a contradiction
  that $\aleph_{0} < |B| < \cc$.

  Let $Q \seb B$ be countable and (order) dense in $B$, i.e.,
  for all $x < y$ in $B$, $[x,y] \cap Q \neq \es$. Since
  $B$ is uncountable, $Q$ must contain a copy of $\mathbb{Q}$.
  Since $\mathbb{Q}$ has $\cc$ gaps, and $Q$ is countable,
  we deduce that $|G(Q)| = \cc$.

  For every $b \in B \setm Q$ let
  $g(b) := (\{q \in Q : q \le b\},\{q \in Q : b \le q\})$.
  Since $|B| < \cc$, $|G(Q) \setm \ran(g)| = \cc$,
  and it is clear that any $(X,Y) \in G(Q) \setm \ran(g)$
  maps to a gap in $B$, namely
  $(\{b \in B : b \le \inf(Y)\},\{b \in B : \sup(X) \le b\})$,
  and that this applications is injective. Thus
  $\cc \le |G(B)|$, a contradiction.
\end{proof}

We now describe the way to choose the $A_{\alpha}$.
At step $\alpha$ we add elements to $A_{\alpha}$
according to the following:
\begin{enumerate}\itemsep0em
  \item For each $\xi < \alpha$ we ask if
  $|\mathbb{R}\setm \dom(f_{\xi})| = \cc$. If the answer is yes,
  we add any element of $(\mathbb{R}\setm \dom(f_{\xi})) \setm Z_{\alpha}$
  to $A_{\alpha}$.
  \item For each $\xi < \alpha$ we ask if $|\ran(f_{\xi})| = \cc$.
  If the answer is yes, we add to $A_{\alpha}$ any
  $a \in \dom(f_{\xi}) \setm Z_{\alpha}$ such that $f_{\xi}(a) \notin Z_{\alpha}$.
\end{enumerate}
If neither step (1) or (2) adds elements to $A_{\alpha}$, simply let
$A_{\alpha}$ be singleton and disjoint from $Z_{\alpha}$. Note that clearly
$|A_{\alpha}| \le \max\{\aleph_{0},|\alpha|\}$.

\begin{lemma}
  Let $X \seb \cc$ be such that $|X| = \cc$, and
  $B \seb \mathbb{R}$ be such that $\aleph_{0} < |B| < \cc$.
  Then $A(X) \nsgeq B$.
\end{lemma}
\begin{proof}
  Suppose $A(X) \sgeq B$ and that $\aleph_{0} < |B| < \cc$. Let $f : A(X) \sur B$ be
  an epimorphism, and find $\alpha$ such that
  $f \seb f_{\alpha}$.

  First note that $|\mathbb{R} \setm \dom(f_{\alpha})| = \cc$ is impossible:
  suppose this is the case, then taking
  $\beta > \alpha$ in $X$, we see that the step (1) of the construction implies
  that $A(X)$ contains some $a \notin \dom(f_{\alpha})$,
  but then also $a \notin \dom(f)$ which is a contradiction. So,
  $|\mathbb{R} \setm \dom(f_{\alpha})| < \cc$, and thus by
  Proposition~\ref{prop:gaps}, and using the fact that $B$ is uncountable,
  we conclude that $|\ran(f_{\alpha})| = \cc$.

  Let $X' := X \setm (\alpha+1)$, and note that by hypothesis,
  $|X'| = \cc$. Now, for each $\beta \in X'$ pick
  $a_{\beta} \in A_{\beta} \cap \dom(f_{\alpha})$
  such that $f_{\alpha}(a_{\beta}) \notin Z_{\beta}$. This exists by
  step (2) of the construction. We claim that if $\beta < \gamma$
  are in $X'$, then $f_{\alpha}(a_{\beta}) \neq f_{\beta}(a_{\gamma})$.
  This finishes the proof because then
  $|f(A(X))| = |f_{\alpha}(A(X))| \ge |X'| = \cc > |B|$, which
  is a contradiction.

  To see this, note that since $\alpha < \beta < \gamma$,
  $a_{\beta} \in A_{\beta} \seb Z_{\gamma}$, thus $f_{\alpha}(a_{\beta}) \in Z_{\gamma}$,
  and by construction $f_{\alpha}(a_{\gamma}) \notin Z_{\gamma}$.
\end{proof}

\section{On a question on Countryman lines}\label{sec:irreversible}

In this section we take a detour from epimorphisms, and study
the following question on Countryman lines, which in our opinion is quite
natural, and to the best of our knowledge, has not been considered
in the literature.
Recall that if $C$ is a Countryman line then it enjoys the property
of having no two uncountable reverse isomorphic suborders.
We ask if this characterizes the Countryman lines in the class
of Aronszajn lines.

Usually a linear order $L$ is called \emph{reversible} if
$L \cong L^{\star}$, otherwise it is called irreversible. In
this spirit we make the following definition.

\begin{definition}\label{df:sirreversible}
  Let $L$ be a linear order. For a cardinal $\kappa$ we say that $L$ is
  \emph{$\kappa$-irreversible} if $|L| \ge \kappa$ and
  $L$ contains no two uncountable reverse
  isomorphic suborders of cardinality $\kappa$.
\end{definition}

Observe that the $\omega$-irreversible linear orders
are exactly the infinite ordinals and their reverses.
Easy examples of $\omega_{1}$-irreversible uncountable orders are
$\omega_{1}$, $\omega_{1}^{\star}$ and Countryman lines. As mentioned,
we are interested in the following question. Is there a non-Countryman
$\omega_{1}$-irreversible Aronszajn line?
We have not seen an example of such
a line in the literature.

Moore~\cite{Moore2009} proved that under
$\PFA$, every non-Countryman Aronszajn line contains both orientations
of any given Countryman line. Thus under $\PFA$ there is no such
example. It is well known that under $\PFA$ there are no Suslin trees, and that
any lexicographic ordering of a Suslin tree is a non-Countryman
Aronszajn line. Thus it is natural to look for an example that is
the lexicographic ordering of a Suslin tree.

Let $\mathbb{P}$ be the forcing for adding a single Cohen real. Here
we take it as the finite partial functions from $\omega$ to itself,
ordered by reverse inclusion. It is well known that if $G$ is a generic
filter for $\mathbb{P}$, then in $V[G]$ there is a Suslin tree. The
original proof is due to Shelah (see~\cite{Shelah1984}). We will
prove that in this same model there is an example of an Aronszajn
line with the desired properties.

\begin{theorem}\label{thm:sro_cohen_real}
  If $\mathbb{P}$ is the forcing for adding a single Cohen real, then in
  $V[G]$ there is a $\omega_{1}$-irreversible lexicographically ordered Suslin
  tree.
\end{theorem}

Our proof will be a modest modification of Todorcevic's proof of Shelah's theorem
(see~\cite{Todorcevic1987}). For details
see~\cite[Proposition~20.7]{Halbeisen2017}.

Fix any uncountable and coherent $T \seb \omega^{<\omega_{1}}$ of finite-to-one
functions such as the one used in the proof of Theorem~\ref{thm:L_R}. As
mentioned in Lemma~\ref{lem:coherent-fin-countryman}, the lexicographic ordering
of $T$ is a Countryman line.
We also need the following lemma, for a proof see the claim
of the mentioned proposition in~\cite{Halbeisen2017}.

\begin{lemma}\label{lem:halbeisen}
  Let $K \seb T$ be uncountable, and let $n \in \omega$. Then
  there is an uncountable $K' \seb K$ such that
  for all $s,t \in K'$, if $\xi \in \dom(s) \cap \dom(t)$,
  and $t(\xi) < n$ or $s(\xi) <n$, then $t(\xi) = s(\xi)$.
\end{lemma}

We turn now to prove Theorem~\ref{thm:sro_cohen_real}.
Fix $G$ a generic filter for $\P$, and let $\dt{c}$ be the canonical
name for $\bigcup G$. Then clearly $\1 \Vdash \dt{c} : \ck{\omega} \to \ck{\omega}$.
Now working in $V[G]$ let $T^{c} := \{c \circ t : t \in T\}$. Todorcevic
proved that this is always a Suslin tree. Also a proof of this
is easily extracted from what follows, and not by chance, but because
the proof here follows that one.

\begin{proposition}
  $V[G] \models$ $(T^{c},<_{\lex})$ is $\omega_{1}$-irreversible.
\end{proposition}
\begin{proof}
  Suppose towards a contradiction
  that $V[G] \models T^{c} \text{ is not $\omega_{1}$-irreversible}$.
  Then there must be $p \in G$ and names $\dt{A}$,$\dt{f}$ such
  that
  \[p \Vdash \dt{A} \text{ is uncountable and }
    f : \dt{A} \seb T^{\dt{c}} \to T^{\dt{c}} \text{ reverses } <_{\lex}.\]
  To arrive to a contradiction it is enough to find some extension
  of $p$ that forces
  $(\exists t,s \in \dt{A})(t <_{\lex} s \land \dt{f}(t) <_{\lex} \dt{f}(s))$.

  For each $q \le p$, consider the
  set $A_{q} := \{t \in T : q \Vdash \dt{c} \circ \ck{t} \in \dt{A}\}$. Since $\P$ is
  countable, there is some $q \le p$ such that $A_{q}$ is uncountable. Fix such a $q$.
  In similar fashion, there is $q' \le q$, an uncountable $K \seb A_{q}$,
  and $g : K \to T$ in $V$ such that for all $t \in K$,
  $q' \Vdash \dt{f}(\dt{c} \circ \ck{t}) = \dt{c} \circ \ck{g}(\ck{t})$. We may
  assume that this is already forced by $q$. Let $n$ be greater
  than any element in $\dom(q)$. By going to an uncountable
  subset of $K$, we may assume that $K$ and $\{g(t) : t \in K\}$ satisfy
  the conclusions of Lemma~\ref{lem:halbeisen}.

  Since $T$ is Countryman we see that $T$ is
  $\omega_{1}$-irreversible, and thus there are $t,s \in K$
  such that $t <_{\lex} s$ and $g(t) <_{\lex} g(s)$. There are four
  cases depending on the extension relation of $t$ with $s$ and
  of $g(t)$ with $g(s)$. In
  case that $t \not\tle s$ we let $x$ and $y$ be such that
  $t(\xi) = x$ and $s(\xi) = y$, where $\xi$ is the first such that
  $t(\xi) \neq s(\xi)$. Otherwise we let $x$ and $y$ undefined.
  In similar fashion, when $g(t) \not\tle g(s)$
  we let $a = g(t)(\xi)$ and $b = g(s)(\xi)$ where
  $\xi$ is the first such that $g(t)(\xi) \neq g(s)(\xi)$.
  Note that when they are defined,
  $a,b,x,y > n$, $x < y$ and $a < b$.

  If $t \tle s$ and $g(t) \tle g(s)$, then clearly
  $q \Vdash \dt{c} \circ \ck{t} \tle \dt{c} \circ \ck{s}$ and
  $q \Vdash \dt{c} \circ \ck{g}(\ck{t}) \tle \dt{c} \circ \ck{g}(\ck{s})$.
  This arrives gives us our desired contradiction, since
  $q \Vdash \dt{c} \circ \ck{t}, \dt{c} \circ \ck{s} \in \dt{A}$ and
  $q \Vdash \dt{c} \circ \ck{g}(\ck{t}) = \dt{f}(\dt{c} \circ \ck{t}) \land
  \dt{c} \circ \ck{g}(\ck{s}) = \dt{f}(\dt{c} \circ \ck{s})$.

  If $t \tle s$ and $g(t) \not\tle g(s)$, then
  letting $q' := q \cup \{(a,0),(b,1)\}$ we see that $q' \in \P$,
  $q' \le q$ and
  $q' \Vdash \dt{c} \circ \ck{t} \tle \dt{c} \circ \ck{s} \land
  \dt{c} \circ \ck{g}(\ck{t}) <_{\lex} \dt{c} \circ \ck{g}(\ck{s})$,
  which is again a contradiction.
  If $t \not\tle s$ and $g(t) \tle g(s)$ the argument is identical.

  Assume now that $t \not\tle s$ and $g(t) \not\tle g(s)$.
  If $x = b$,
  then $a < x = b < y$ and thus
  $q' := q \cup \{(a,0),(x,1),(y,2)\}$ is in $\P$. But note then
  $q' \le q$ and
  $q \Vdash \dt{c} \circ \ck{t} <_{\lex} \dt{c} \circ \ck{s}
  \land \dt{c} \circ \ck{g}(\ck{t}) <_{\lex} \dt{c} \circ \ck{g}(\ck{s})$,
  and we arrive at a contradiction as before. Assume now $x \neq b$.
  If $y = a$
  then $q' := q \cup \{(x,0),(a,1),(b,2)\}$ works, and if $y \neq a$
  then $q' := q \cup \{(x,0),(y,1),(a,0),(b,1)\}$ works.
\end{proof}

\begin{remark}\label{rmk:coherent-irrev}
  It was pointed out to us by Moore (and then also by the anonymous referee)
  that that any coherent Aronszajn
  tree $T \seb \omega^{<\omega_{1}}$ is $\omega_{1}$-irreversible
  with the natural lexicographic ordering. Note that this
  gives an alternative proof of Theorem~\ref{thm:sro_cohen_real}, and it also
  shows that $\diamondsuit$ implies the existence of an
  $\omega_{1}$-irreversible lexicographically ordered Suslin
  tree. This is because $\diamondsuit$ implies the existence of a
  coherent Suslin tree (see for
  example~\cite[Theorem~6.9]{Todorcevic1984}).

  To see that the claim is true one proves that the ccc forcing for
  specializing $T$ (see~\cite[Theorem~16.17]{Jech2003}) makes
  $(T,<_{\lex})$ Countryman in the extension
  (see~\cite[4.4]{TodorcevicWalks}).  Thus in the extension
  $(T,<_{\lex})$ has no two uncountable reverse isomorphic
  suborders. But since $\omega_{1}$ is preserved, this must already hold
  in the ground model.
\end{remark}

\section{Concluding remarks and questions}

As mentioned in Section~\ref{sec:srs}, this work leaves open
 Question~\ref{q:univ-srs} of whether $\eta_{C}$ (or any other universal Aronszajn line)
is strongly surjective under $\PFA$.
Recently in \cite{EisworthEtAl2023}, answering a question by
Baumgartner, a $\eleq$-minimal Countryman line was constructed under
$\diamondsuit$. As mentioned in the introduction,
in \cite{Soukup} it is proved that $\diamondsuit^{+}$ implies
the existence of a strongly surjective Suslin line. Does
$\diamondsuit$ imply the existence of a strongly surjective Aronszajn
line?

The work in Sections~\ref{sec:infinite-antichain} and \ref{sec:infinite-chain} shows
that already the class of Countryman lines is not well-quasi-ordered by
$\sleq$. This is done by considering Aronszajn lines that have
decompoisitions with a stationary set of endpoints. It seems that a natural
question is whether the class of \textit{normal} Aronszajn lines is
well-quasi-ordered by $\sleq$ under $\PFA$. By Theorem~\ref{thm:moore-t1.1}
this is clearly true for the class of Countryman lines.

Regarding Section~\ref{sec:irreversible}, Theorem~\ref{thm:sro_cohen_real} and
 Remark~\ref{rmk:coherent-irrev} suggest the following question. Does the
existence of a Suslin tree suffices to show the existence of a
non-Countryman $\omega_{1}$-irreversible Aronszajn line?

\bibliographystyle{plain}

\bigskip
\noindent Carlos A. Mart\'inez--Ranero\\
Email: cmartinezr@udec.cl\\
Homepage: www2.udec.cl/\textasciitilde cmartinezr\\
\noindent Lucas Polymeris\\
Email: l.polymeris@proton.me \\

\noindent Same address: \\
Universidad de Concepci\'on, Concepci\'on, Chile\\
Facultad de Ciencias F\'isicas y Matem\'aticas\\
Departamento de Matem\'atica\\

\end{document}